\documentclass[MSNbibl,number,citesort,seceqn,dvips]{arxbj}
\usepackage{upgreek}
\usepackage{graphicx}

\aid{0}
\volume{20}
\issue{3}
\pubyear{2014}
\firstpage{1372}
\lastpage{1403}
\doi{10.3150/13-BEJ525} %kopijuoti is PTS

\makeatletter
\newcommand{\iint}{\int\!\!\int}

\newcommand{\RMO}{\mathrm{O}}
\newcommand{\mrmi}{\mathrm{i}}

\newcommand{\mrme}{\mathrm{e}}
\newcommand{\mrmd}{\,\mathrm{d}}
\newcommand{\mrmdd}{\mathrm{d}}

\newproclaim{Assumption}{Assumption}
\newcommand{\rrvert}{\vert}
\newcommand{\llvert}{\vert}

\newtheorem{lemma}{Lemma}[section]
\newtheorem{theorem}[lemma]{Theorem}
\newtheorem{proposition}[lemma]{Proposition}
\newtheorem{corollary}[lemma]{Corollary}
\newproclaim{definition}{Definition}
\newremark{remark}[lemma]{Remark}

\newcommand{\ball}{\mathrm{B}}
\newcommand{\ellip}{\mathrm{E}}

\newcommand{\R}{{\mathbb R}}
\newcommand{\N}{{\mathbb N}}
\newcommand{\Z}{{\mathbb Z}}
\newcommand{\A}{{\mathcal A}}
\newcommand{\B}{{\mathcal B}}
\newcommand{\F}{{\mathcal F}}
\newcommand{\G}{{\mathcal G}}
\newcommand{\X}{{\mathcal X}}
\newcommand{\E}{{\mathbb E}}
\newcommand{\Cov}{\operatorname{Cov}}
\newcommand{\cla}{{\stackrel{\mathcal{D}}{\longrightarrow}}}
\newcommand{\dconv}{\stackrel{\mathcal{D}}{\longrightarrow}}
\newcommand{\lra}{\longrightarrow}

\renewcommand{\epsilon}{\varepsilon}
\renewcommand{\P}{{P}}
\renewcommand{\phi}{\varphi}
\renewcommand{\H}{{\mathcal H}}

\makeatother

\begin{document}
\begin{frontmatter}

\title{Approximating class approach for empirical processes of
dependent sequences indexed by~functions}
\runtitle{Approximating class approach for empirical processes}

\begin{aug}
%%%% inicialai - be tarpu
\author[1]{\inits{H.}\fnms{Herold} \snm{Dehling}\thanksref{1,e1}\ead[label=e1,mark]{herold.dehling@rub.de}},
\author[2]{\inits{O.}\fnms{Olivier} \snm{Durieu}\corref{}\thanksref{2}\ead[label=e2]{olivier.durieu@lmpt.univ-tours.fr}} \and
\author[1]{\inits{M.}\fnms{Marco} \snm{Tusche}\thanksref{1,e3}\ead[label=e3,mark]{marco.tusche@rub.de}}
\runauthor{H. Dehling, O. Durieu and M. Tusche} %% auto
\address[1]{Fakult\"at f\"ur Mathematik, Ruhr-Universit\"at Bochum,
44780 Bochum, Germany.\\ \printead{e1,e3}}
\address[2]{Laboratoire de Math\'ematiques et Physique Th\'eorique
UMR-CNRS 7350,
F\'ed\'eration Denis Poisson FR-CNRS 2964,
Universit\'e Fran\c cois-Rabelais de Tours, Parc de Grandmont, 37200 Tours,
France.
\printead{e2}}
\end{aug}

% HISTORY:
\received{\smonth{2} \syear{2012}}
\revised{\smonth{11} \syear{2012}}

% ABSTRACT
%
\begin{abstract}
We study weak convergence of empirical processes of dependent data
$(X_i)_{i\geq0}$, indexed by classes of functions. Our results are
especially suitable for data arising from dynamical systems and Markov
chains, where the central limit theorem for partial sums of observables
is commonly derived via the spectral gap technique. We are
specifically interested in situations where the index class $\F$ is
different from the class of functions $f$ for which we have good
properties of the observables $(f(X_i))_{i\geq0}$. We introduce a new
bracketing number to measure the size of the index class $\F$ which
fits this setting. Our results apply to the empirical process of data
$(X_i)_{i\geq0}$ satisfying a multiple mixing condition. This includes
dynamical systems and Markov chains, if the Perron--Frobenius operator
or the Markov operator has a spectral gap, but also extends beyond this
class, for example, to ergodic torus automorphisms.
\end{abstract}

% KEYWORDS
% visi is mazosios raides ir pagal abecele
%

\end{frontmatter}

%s1 #&#
\section{Introduction}

Let $(X_i)_{i\geq0}$ be a stationary stochastic process of $\R$-valued
random variables with marginal distribution $\mu$. We denote the
empirical measure of order $n$ by $\mu_n=\frac{1}{n}\sum_{i=1}^n
\delta_{X_i}$. The classical empirical process is defined by
$U_n(t)=\sqrt{n}(\mu_n((-\infty,t])-\mu((-\infty,t]))$, $t\in\R
$. In
the case of i.i.d. processes, the limit behavior of the empirical
process was first investigated by Donsker \cite{Don52}, who proved that
$(U_n(t))_{t\in\R}$ converges weakly to a Brownian bridge process.
This result, known as Donsker's empirical process central limit
theorem, confirmed a conjecture of Doob \cite{Doo49} who had observed that
certain functionals of the empirical process converge in distribution
towards the corresponding functionals of a Brownian bridge. Donsker's
empirical process CLT has been generalized to dependent data by many
authors. One of the earliest results is Billingsley \cite{Bil68}, who
considered
functions of mixing processes, with an application to the empirical
distribution of the remainders in a continued fraction expansion.

Empirical processes play a very important role in large sample
statistical inference. Many statistical estimators\vadjust{\goodbreak
} and test statistics
can be expressed as functionals of the empirical distribution. As a
result, their asymptotic
distribution can often be derived from empirical process limit
theorems, combined with the continuous mapping theorem or a functional
delta method. A well-known example is the Kolmogorov--Smirnov
goodness-of-fit test, which uses the test statistic $D_n:=\sup_{t\in
\R}\sqrt{n} |\mu_n((-\infty,t])-
\mu_0((-\infty,t])|$ in order to test the
null hypothesis that $\mu_0$ is the marginal distribution of $X_1$.
Under the null hypothesis, the limit
distribution of $D_n$ is given by the supremum of the Gaussian limit of
the empirical process. Another example are Von-Mises-statistics, also
known as V-statistics. These are defined as
$V_n:=\frac{1}{n^2}\sum_{1\leq i,j\leq n} h(X_i,X_j)$, where $h(x,y)$
is a symmetric kernel function.
Specific examples include the sample variance and Gini's mean
difference, where the kernel functions are
given by
$(x-y)^2/2$ and $|x-y|$, respectively. $V$-statistics can be expressed
as integrals with respect to the empirical distribution function, namely
$V_n=\iint h(x,y) \mrmd\mu_n(x) \mrmd\mu_n(y) $. The asymptotic
distribution of $V_n$ can then be derived via a functional delta method
from an empirical process central limit theorem; see, for example,
Beutner and Z{\"a}hle \cite{BeuZah12} for some recent results.

Empirical process CLTs for $\R^d$-valued i.i.d. data $(X_i)_{i\geq0}$
have first been studied by Dudley~\cite{Dud66}, Neuhaus \cite{Neu71},
Bickel and Wichura \cite{BicWic71} and Straf \cite{Str72}. These
authors consider the classical $d$-dimensional empirical process
$\sqrt{n}( \mu_n((-\infty,t]) -\mu((-\infty,t]))$, where
$(-\infty,t]=\{x\in\R^d\dvt x_1\leq t_1,\ldots, x_d\leq t_d\}$, $t
\in\R^d$, denotes the semi-infinite rectangle in $\R^d$. Philipp and
Pinzur~\cite{PhiPin80}, Philipp \cite{Phi84} and Dhompongsa
\cite{Dho84} studied weak convergence of the multivariate empirical
process in the case of mixing data.

Dudley \cite{Dud78} initiated the study of empirical processes indexed by
classes of sets, or more generally by classes of functions. This
approach allows the study of empirical processes for very general data,
not necessarily having values in Euclidean space. CLTs for empirical
processes indexed by classes of functions require entropy conditions on
the size of the index set. For i.i.d. data, Dudley \cite{Dud78} obtained
the CLT for
empirical processes indexed by classes of sets satisfying an entropy
condition with inclusion. Ossiander \cite{Oss87} used an entropy
condition with
bracketing to obtain results for empirical processes indexed by classes
of functions. For the theory of empirical processes of i.i.d. data,
indexed by classes of functions, see the book by van der Vaart and
Wellner~\cite{VanWel96}.
Limit theorems for more general empirical processes indexed by classes
of functions have also been studied under entropy conditions for
general covering numbers, for example, by Nolan and Pollard \cite
{NolPol87} who
investigate empirical $U$-processes.

In the case of strongly mixing data, Andrews and Pollard \cite
{AndPol94} were the first
to obtain CLTs for empirical processes indexed by classes of functions.
Doukhan, Massart and Rio \cite{DouMasRio95} and Rio \cite{Rio98} study
empirical processes for
absolutely regular data.
Borovkova, Burton and Dehling
%Andrews and Pollard
\cite{BorBurDeh01}
investigate the empirical
process and the empirical $U$-process for data that can be represented
as functionals of absolutely regular processes. For further results,
see the survey article by Dehling and Philipp~\cite{DehPhi02}, the book
by Dedecker \textit{et al.} \cite{DedDouLan07}, as well as the paper by
Dedecker and Prieur \cite{DedPri07}.

A lot of research has been devoted to the study of statistical
properties of data arising from dynamical systems or from Markov
chains. A very powerful technique to prove CLTs and other limit
theorems is the spectral gap method, using spectral properties of the
Perron--Frobenius operator or the Markov operator on an appropriate
space of functions; see Hennion and Herv{\'e}~\cite{HenHer01}. When the
space of functions under consideration contains the class of indicator
functions of intervals, standard tools can be used to establish the
classical empirical process CLT. Finite-dimensional convergence of the
empirical process follows from the CLT\vadjust{\goodbreak} for
$\sum_{i=1}^n 1_{(-\infty,t]}(X_i)$, and tightness can be established
using moment bounds for $\sum_{i=1}^n 1_{(s,t]}(X_i)$.
%Andrews and
%Pollard
Collet, Martinez and Schmitt
\cite{ColMarSch04} used this approach to establish the
empirical process CLT for expanding maps of the unit interval.

The situation differs markedly when the CLT and moment bounds are not
directly available for the index class of the empirical process, but
only for a different class of functions. Recently, Dehling, Durieu and
Voln\'y \cite{DehDurVol09} developed techniques to cover such situations.
They were able to prove classical empirical process CLTs for
$\R$-valued data when the CLT and moment bounds are only available for
Lipschitz functions. Dehling and Durieu \cite{DehDur11} extended these
techniques to $\R^d$-valued data satisfying a multiple mixing
condition for H\"older continuous functions. Under this condition, they
proved the CLT for the empirical process indexed by semi-infinite
rectangles $(-\infty,t]$, $t\in\R^d$. The multiple mixing condition is
strictly weaker than the spectral gap condition. For example, ergodic
torus automorphisms satisfy a multiple mixing condition, while
generally they do not have a spectral gap. Dehling and Durieu
\cite{DehDur11} proved the empirical process CLT for ergodic torus
automorphisms. Durieu and Tusche \cite{DurTus12} provide very general
conditions under which the classical empirical process CLT for
$\R^d$-valued data holds.

The above mentioned papers study exclusively classical empirical
processes, indexed by semi-infinite intervals or rectangles. It is
the goal of the present paper to extend the techniques developed by
Dehling, Durieu and Voln\'y \cite{DehDurVol09} to empirical processes
indexed by classes of functions.
Let $(\X,\A)$ be a measurable space, let $(X_i)_{i\geq0}$ be a
stationary process of $\X$-valued random
variables, and let $\F$ be a uniformly bounded class of real-valued
functions on $\X$. We consider
the $\F$-indexed empirical process
$(\frac{1}{\sqrt{n}} \sum_{i=1}^n (f(X_i)-\E f(X_1)))_{f\in\F}$.
As in the above mentioned papers, we will assume that there exists some
Banach space $\B$ of functions on $\X$ such that the CLT and a moment
bound hold for partial sums $\sum_{i=1}^n g(X_i)$, for all $g$ in some
subset of $\B$; see Assumptions \ref{assump1} and \ref{assump2}.
These conditions
are satisfied, for example, when the Perron--Frobenius operator or the
Markov operator acting on $\B$ has a spectral
gap. Again, if the index class $\F$ is a subset of $\B$, standard
techniques for proving empirical process
CLTs can be applied. In many examples, however, $\B$ is some class of
regular functions, while $\F$ is a class of indicators of sets. It is
the goal of the present paper to provide techniques suitable for this situation.

Empirical process invariance principles require a control on the size
of the index class $\F$, as measured
by covering or bracketing numbers; see, for example, van der Vaart and
Wellner~\cite{VanWel96}. In this paper, we will
consider coverings of $\F$ by $\B$-brackets, that is, brackets
bounded by functions $l,u\in\B$. Because
of the specific character of our moment bounds, we have to impose
conditions on the $\B$-norms of $l$ and $u$. We will thus introduce a
notion of bracketing numbers by counting how many $\B$-brackets of a
given $L^s$-size and with a given control on the $\B$-norms of the
upper and lower functions are
needed to cover $\F$. The main theorem of the present paper
establishes an empirical process CLT under
an integral condition on this bracketing number.

This paper is organized as follows: Section \ref{secmain} contains
precise definitions as well as the statement of the main theorem. In
Section \ref{secexamples}, we will specifically consider the case
when $\B$ is the space of H\"older
continuous functions. We will give examples of classes of functions
which satisfy the bracketing number assumption. In Section \ref
{secappli}, we will give applications to ergodic torus automorphisms
which extend the
empirical process CLT of Dehling and Durieu \cite{DehDur11} to more
general classes of
sets. Section \ref{secpf-ep-clt} contains the proof of our main
theorem, while proofs of technical aspects of the examples can be found
in the \hyperref[app]{Appendix}.

%s2 #&#
\section{Main result}\label{secmain}

Let $(\X,\A)$ be a measurable space, and let $(X_i)_{i\in\N}$ be an
$\X$-valued stationary stochastic
process with marginal distribution $\mu$. Let $\F$ be a uniformly
bounded class of real-valued measurable functions defined on $\X$. If
$Q$ is a signed measure on $(\X,\A)$, we use the notation $Qf=\int
_\X
f\mrmd Q$. We define the map $F_n\dvtx \F\rightarrow\R$, induced by the
empirical measure,
\[
F_n(f)=\frac{1}{n} \sum_{i=1}^n
f(X_i).
\]
The $\F$-indexed empirical process of order $n$ is given by
\[
U_n(f)=\sqrt{n} \bigl(F_n(f)-\mu f\bigr) =
\frac{1}{\sqrt{n}} \sum_{i=1}^n
\bigl(f(X_i)-\mu f\bigr),\qquad f\in\F.
\]
We regard the empirical process $(U_n(f))_{f\in\F}$ as a random element
on $\ell^\infty(\F)$; this holds as
$\F$ is supposed to be uniformly bounded. $\ell^\infty(\F)$ is
equipped with the supremum norm and the
Borel $\sigma$-field generated by the open sets. It is well known
that, in general, $(U_n(f))_{f\in\F}$ is not
measurable and thus the usual theory of weak convergence of random
variables does not apply. We use here
the theory which is based on convergence of outer expectations; see
van der Vaart and Wellner \cite{VanWel96}. Given a Borel probability
measure $L$ on $\ell
^\infty(\F)$, we say that $(U_n(f))_{n\geq1}$ converges in
distribution to
$L$ if
\[
\E^\ast\bigl(\phi(U_n)\bigr)\rightarrow\int\phi(x) \mrmd L(x)
\]
for all bounded and continuous functions $\phi\dvtx\ell^\infty(\F)
\rightarrow\R$. Here $\E^\ast$ denotes the outer integral. Note
that $\E^\ast(X)=\E(X^\ast)$, where $X^\ast$ denotes the
measurable cover function of $X$; see Lemma~1.2.1 in van der Vaart and
Wellner \cite{VanWel96}.

In what follows, we will frequently make two assumptions concerning the
process $(f(X_i))_{i\in\N}$,
where $f\dvtx\X\rightarrow\R$ belongs to some Banach space $(\B,\|
\cdot
\|_\B)$ of measurable
functions on $\X$,
respectively, to some subset $\G\subset\B$. The precise choice of
$\B
$, as well as of $\G$,
will depend on the specific example. Often, we take $\B$ to be the
space of all Lipschitz or H\"older continuous functions, and $\G$ the
intersection of $\B$ with an $\ell^\infty(\X)$-ball.

\begin{Assumption}[(CLT for $\B$-observables)]\label{assump1}
For all $f\in\B$, there exists a $\sigma_f^2\geq0$ such that
%
%e2.1 #&#
%
\begin{equation}
\label{clt} \frac{1}{\sqrt{n}}\sum_{i=1}^n
\bigl(f(X_i)-\mu f\bigr) \cla N\bigl(0,\sigma_f^2
\bigr),
\end{equation}
where $N(0,\sigma^2)$ denotes the normal law with mean zero and
variance $\sigma^2$.
\end{Assumption}

\begin{Assumption}[(Moment bounds for $\bolds{\G
}$-observables)]\label{assump2}
For some subset $\G\subset\B$, $s\ge1$, and $a\in\R$, for all
$p\geq1$, there exists a constant $C_p>0$ such that for all $f\in\G
-\G:=\{g_1-g_2\dvt g_1,g_2\in\G\}$,
%
%e2.2 #&#
%
\begin{equation}
\label{2pbound} \E\Biggl[ \Biggl(\sum_{i=1}^n
\bigl(f(X_i)-\mu f\bigr) \Biggr)^{2p} \Biggr] \leq
C_p\sum_{i=1}^p
n^i \|f\|_s^i \log^{2p+ai} \bigl(\|f
\|_\B+1\bigr),
\end{equation}
where $\|f\|_s=(\int_\X|f|^s \mrmd\mu)^{1/s}$ denotes the
$L^s$-norm of $f$.
\end{Assumption}

Both Assumptions \ref{assump1} and \ref{assump2} have been
established by many
authors for a wide range of stationary processes. Concerning the CLT,
see, for example, the three-volume monograph by Bradley \cite{Bra07} for
mixing processes, Dedecker \textit{et al.} \cite{DedDouLan07} for
so-called weakly dependent
processes in the sense of Doukhan and Louhichi \cite{DouLou99}, and
Hennion and Herv{\'e} \cite{HenHer01} for
many examples of Markov chains and dynamical systems. Durieu \cite{Dur08}
proved $4$th moment bounds of the type (\ref{2pbound}) for Markov
chains or dynamical systems for which the Markov operator or the
Perron--Frobenius operator acting on $\B$ has a spectral gap. It was
generalized to $2p$th moment bounds by Dehling and Durieu \cite
{DehDur11}. More
generally, they gave similar moment bounds for processes satisfying a
multiple mixing condition, that is, assuming that there exist a
$\theta\in(0,1)$ and an integer $d_0\in\N$ such that for all integers
$p\geq1$, there exist an integer $\ell$ and a multivariate polynomial
$P$ of total degree smaller than $d_0$ such that
%
%e2.3 #&#
%
\begin{eqnarray}
\label{mm}
&&
\bigl\llvert\Cov\bigl(f(X_{i_0})\cdots
f(X_{i_{q-1}}),f(X_{i_q})\cdots f(X_{i_p}) \bigr) \bigr\rrvert
\nonumber\\[-8pt]\\[-8pt]
&&\quad\leq\|f
\|_s \|f\|_{\B}^\ell
P(i_1-i_0,\ldots,i_p-i_{p-1})\theta^{i_q-i_{q-1}}\nonumber
\end{eqnarray}
holds for all $f\in\B$ with $\mu f=0$ and $\|f\|_\infty\leq1$, all
integers $i_0\le i_1 \le\cdots\le i_p$ and all $q\in\{1,\ldots,p\}$.
See Theorem 4 and the examples in Dehling and Durieu \cite{DehDur11}.
Note that this multiple mixing condition implies the moment bound
(\ref{2pbound}) with for $\G=\{f\in\B\dvt \|f\|_\infty\leq1\}$ and
$a=d_0-1$. Further, the spectral gap property leads to the multiple
mixing condition with $d_0=0$, and thus to the moment bound
(\ref{2pbound}) with $a=-1$, see Dehling and Durieu \cite{DehDur11},
Section 4.

We will derive a general statement about weak convergence of
the empirical process $(U_n(f))_{f\in\F}$ under the two assumptions
(\ref{clt}) and (\ref{2pbound}). Empirical process central limit
theorems require bounds on the size of the class of functions $\F$,
usually measured by the number of $\epsilon$-balls required to cover
$\F$. Here we will introduce a covering number adapted to the fact that
(\ref{clt}) and (\ref{2pbound}) hold only for $f\in\B$ or $f\in\G
$, respectively, and that both the $\B$-norm as well as the $L^s(\mu
)$-norm enter on the right hand side of the bound (\ref{2pbound}). In
our approach, we use $\B$-brackets to cover the class $\F$, which
leads to the following definition.

\begin{definition*}\label{defbracket}
Let $(\X,\A)$ be a measurable space, and let $\mu$ be a probability
measure on $(\X,\A)$.
Let $\B$ be some Banach space of measurable functions on $\X$, $\G
\subset\B$ and $s\geq1$.

\begin{longlist}[(ii)]
\item[(i)]
Given two functions $l,u\dvtx\X\rightarrow\R$ satisfying $l(x)\leq
u(x)$, for all $x\in\X$, we define the bracket
\[
[l,u]:=\bigl\{f\dvtx\X\rightarrow\R\dvt l(x) \leq f(x) \leq
u(x)\mbox{, for all } x\in
\X\bigr\}.
\]
Given $\epsilon, A>0$, we call
$[l,u]$ an $(\epsilon,A,\G,L^s(\mu))$-bracket, if $l,u\in\G$ and
\begin{eqnarray*}
\|u-l\|_s&\leq&\epsilon,
\\
\|u\|_\B&\leq& A,\qquad\|l\|_\B\leq A,
\end{eqnarray*}
where $\| \cdot\|_s$ denotes the $L^s(\mu)$-norm.

\item[(ii)] For a class of measurable functions $\F$, defined on $\X
$, we
define the bracketing number $N(\epsilon,A,\F,\G,L^s(\mu))$ as the
smallest number of $(\epsilon, A, \G, L^s(\mu))$-brackets needed to
cover $\F$.
\end{longlist}
\end{definition*}

Our definition is close to the definition of bracketing numbers given
by Ossiander \cite{Oss87}, but different. In Ossiander \cite{Oss87}, no
assumptions are
made on the upper and lower functions of the bracket other than that
they are close in $L^2$. Here, the moment bound (\ref{2pbound}) forces
us to require the extra condition that $u$ and $l$ belong to the space
$\B$ and that their $\B$-norms are controlled. Obviously, our
bracketing numbers are always larger than the ones defined in Ossiander
\cite{Oss87}, and naturally our condition on the size of $\F$
are stronger. On the other hand, our results apply to dependent data,
while Ossiander \cite{Oss87} studies i.i.d. data.

We can now state the main theorem of the present paper.
The proof will be given in Section \ref{secpf-ep-clt}.

%th2.1 #&#
%
\begin{theorem}\label{thep-clt}
Let $(\X,\A)$ be a measurable space, let $(X_i)_{i\geq1}$ be an $\X
$-valued stationary stochastic process with marginal distribution $\mu
$, and let $\F$ be a uniformly bounded class of measurable functions
on $\X$. Suppose that for some Banach space $\B$ of measurable
functions on $\X$, some subset
$\G\subset\B$, $a\in\R$, and $s\geq1$, Assumptions \ref{assump1}
and \ref{assump2} hold.
Moreover, assume that there exist
constants $r>-1$, $\gamma> \max\{2+a,1\}$ and $C>0$ such that
%
%e2.4 #&#
%
\begin{equation}
\label{bracket} \int_0^1 \epsilon^r
\sup_{\epsilon\leq\delta\leq1} N^2\bigl(\delta,\exp\bigl(C
\delta^{-1/\gamma}\bigr), \F,\G,L^s(\mu)\bigr) \mrmd\epsilon
<\infty.
\end{equation}
Then the empirical process $(U_n(f))_{f\in\F}$ converges in
distribution in $\ell^\infty(\F)$ to a tight Gaussian process
$(W(f))_{f\in\F}$.
\end{theorem}
%
%re2.2 #&#
%
\begin{remark}
(i) Note that the bracketing number $N(\delta,\exp(C \delta
^{-1/\gamma}),\F,\G,L^s(\mu))$ might not be a
monotone function of $\delta$. This is the reason why we take the
supremum in the integral (\ref{bracket}).

\mbox{}\hphantom{i}(ii) The proof of Theorem \ref{thep-clt} shows
that the statement
also holds if condition (\ref{2pbound}) is only satisfied for some
integer $p$ satisfying
\[
p>\frac{(r+1) \gamma}{\gamma- \max\{2+a,1\}}.
\]

(iii) If for some $r^\prime\geq0$,
\[
N\bigl(\epsilon, \exp\bigl(C \epsilon^{-1/\gamma}\bigr), \F, \G,
L^s(\mu)\bigr) = \RMO\bigl(\epsilon^{-r^\prime}\bigr)
\]
as $\epsilon\rightarrow0$, condition (\ref{bracket}) is satisfied
for all $r>2r^\prime-1$.
\end{remark}

In the next section, we will present examples of classes of functions
satisfying condition (\ref{bracket}).
Among the examples are indicators of multidimensional rectangles, of
ellipsoids, and of balls of arbitrary metrics, as well as a class of
monotone functions.
In Section \ref{secappli}, we give applications to ergodic torus
automorphisms, indexed by various classes of indicator functions.

%%%%%%%%%%%%%%%%%%%%%%%%%%%%%%%%%%%%%%%%%%%%%%%%%%%%%%%%%%%%%%%%%%%%%%%%%%%%%%%%%%%%%%%%%%%%%%%%%%%%%%%%%%%%%%%%%%%%%%

%s3 #&#
\section{Examples of classes of functions}\label{secexamples}

In many examples that satisfy Assumptions \ref{assump1} and \ref
{assump2}, the Banach space $\B
$ is the space of Lipschitz or H\"older continuous functions, see
examples in Dehling, Durieu and Voln\'y \cite{DehDurVol09}, Dehling and
Durieu \cite{DehDur11}, or Durieu and Tusche \cite{DurTus12}.
Thus, in this section, we will restrict our attention to the case where
$\B$ is a space of H\"older functions and give several examples of
classes $\F$ which satisfy the entropy condition (\ref{bracket}).

In this section, we consider a metric space $(\X,d)$. Let $\alpha\in
(0,1]$ be fixed.
We denote by $\H_{\alpha}(\X)$ the space of bounded $\alpha$-H\"
older continuous functions on $\X$ with values in $\R$. This space is
equipped with the norm
\[
\|f\|_\alpha:= \sup_{x\in\X} \bigl|f(x)\bigr| + \mathop{\sup
_{x,y\in\X
}}_{x\neq
y} \frac{|f(x)-f(y)|}{d(x,y)^\alpha}.
\]
For this section, we chose $\B=\H_\alpha(\X)$. As the approximating
class we use the subclass $\G=\H_\alpha(\X,[0,1]):=
\{f\in\H_\alpha(\X)\dvt 0\leq f\leq1 \}$ of $\B$. Except in Example
\hyperref[exazero-centered-balls]{5}, in all examples we will consider the case
where $\X$ is a subset of $\R^d$ equipped with the Euclidean norm
denoted by $|\cdot|$, where $d\ge1$ is some fixed integer.

In most of the examples, we will use the transition function given in
the following definition which uses the notations
\[
d_A(x):=\inf_{a\in A} d(x,a) \quad\mbox{and}\quad d(A,B):=
\inf_{a\in
A,b\in B} d(a,b)
\]
for any element $x\in\X$ and sets $A$, $B\subset\X$, where we
define $\inf\varnothing=+\infty$.

\begin{definition*}
Let $A$, $B$ be subsets of $\X$ such that $d(A,B)>0$. We define the
transition function $T[A,B]\dvtx\X\rightarrow\R$ by
\[
T[A,B](x):= \frac{d_B(x)}{d_B(x)+d_A(x)},
\]
if $A$ and $B$ are non-empty, $T[A,B]:=0$ if $A=\varnothing$, and
$T[A,B]:=1$ if $B=\varnothing$ but $A\ne\varnothing$.
\end{definition*}
Observe, that we have $T[A,B](\X)\subset[0,1]$, $T[A,B](x)=1$ for all
$x\in A$ and $T[A,B](x)=0$ for all $x\in B$.

%le3.1 #&#
%
\begin{lemma}\label{lemtransition}
For any subsets $A$, $B$ of $\X$ such that $d(A,B)>0$, the transition
function $T[A,B]$ is a bounded $\alpha$-H\"older continuous function
and we have
\[
\bigl\|T[A,B]\bigr\|_\alpha\leq1 + \biggl(\frac{3}{d(A,B)} \biggr
)^{\alpha}.
\]
\end{lemma}

This lemma is proved in the \hyperref[app]{Appendix}.

We also use the following notations: For a non-decreasing function $F$
from $\R$ to $\R$, $F^{-1}$ denotes the pseudo-inverse function
defined by
$F^{-1}(t):=\sup\{x\in\R\dvt F(x)\le t\}$ where \mbox{$\sup
\varnothing
=-\infty$}. The modulus of continuity of $F$ is defined by
\[
\omega_F(\delta)=\sup\bigl\{\bigl|F(x)-F(y)\bigr|\dvt |x-y|\le
\delta\bigr\}.
\]
Constants that only depend on fixed parameters $p_1,\ldots,p_k$ will be
denoted with these parameters in the subscript, such as
$c_{p_1,\ldots,p_k}$. Furthermore, the notation
$f(x)=\RMO_{p_1,\ldots,p_k}(g(x))$ as $x\to0$ or $x\to\infty$ means that
there exists a constant $c_{p_1,\ldots,p_k}$ such that $f(x) \leq
c_{p_1,\ldots,p_k} g(x)$ for all $x$ sufficiently small or large,
respectively.

%%%%%%%%%%%%%%%%%%%%%%%%%%%%%%%%%%%%%%%%%%%%%%%%%%%%%%%%%%%%%%%%%%%%%%%%%%%%%%%%%%%%

%s3.1 #&#
\subsection{Example 1: Indicators of rectangles}
Here, we consider $\X=\R^d$. In its classical form, the empirical
process is defined by the class of
indicator functions of left infinite rectangles, that is, the
class $\{ 1_{(-\infty,t]}\dvt t\in\R^d \}$, where $(-\infty,t]$
denotes the
set of points $x$ such that\footnote{On $\R^d$, we use the partial
order:
$x\le t$ if and only if $x_i\le t_i$ for all $i=1,\ldots,d$.} $x\le t$.
Under similar assumptions as in the present paper, this case was treated
by Dehling and Durieu \cite{DehDur11}. We will see that Theorem \ref{thep-clt}
covers the results of that paper.

The following proposition gives an upper bound for the bracketing number
of the larger class
\[
\F=\bigl\{1_{(t,u]}\dvt t,u\in[-\infty,+\infty]^d, t \le u
\bigr\},
\]
where $(t,u]$ denotes the rectangle which consists of the points $x$
such that
$t<x$ and $x\le u$.

%pr3.2 #&#
%
\begin{proposition}\label{prop-rect}
Let $s\ge1$, $\gamma>1$, and let $\mu$ be a probability distribution
on $\R^d$ whose distribution function
$F$ satisfies
%
%e3.1 #&#
%
\begin{equation}
\label{eq-modulus} \omega_F(x)=\RMO\bigl(\bigl|\log(x)\bigr
|^{-s\gamma}\bigr)
\qquad\mbox{as }x\to0.
\end{equation}
Then there exists a constant $C=C_F>0$ such that
\[
N\bigl(\varepsilon,\exp\bigl(C
\varepsilon^{-{1}/{\gamma}}\bigr),\F,\G,L^s(\mu)\bigr) =\RMO_d
\bigl({\varepsilon^{-2ds}} \bigr) \qquad\mbox{as }\varepsilon
\rightarrow0,
\]
where $\G=\H_\alpha(\R^d,[0,1])$.
\end{proposition}

\begin{pf}
Let $\varepsilon\in(0,1)$ and $m=\lfloor6d\varepsilon^{-s}+1\rfloor
$. For all $i\in\{1,\ldots,d\}$ and $j\in\{0,\ldots,m\}$, we define
the quantiles
\[
t_{i,j}:= F_i^{-1} \biggl(\frac{j}{m}
\biggr),
\]
where $F_i^{-1}$ is the pseudo-inverse of the marginal distribution
function\footnote{$F_i(t)=\mu(\R\times\cdots\times\R\times
(-\infty,t]\times\R\times\cdots\times\R)$.} $F_i$. Now, if
$j=(j_1,\ldots,j_d)\in\{0,\ldots,m\}^d$, we write
\[
t_j=(t_{1,j_1},\ldots,t_{d,j_d}).
\]
In the following definitions, for convenience, we will also denote by
$t_{i,-1}$ or $t_{i,-2}$ the points $t_{i,0}$ and by $t_{i,m+1}$ the
points $t_{i,m}$.
We introduce the brackets $[l_{k,j},u_{k,j}]$, $k\in\{0,\ldots,m\}^d$,
$j\in\{0,\ldots,m\}^d$, $k\le j$, given by the $\alpha$-H\"older functions
\[
l_{k,j}(x):= T \bigl[[t_{k+1},t_{j-2}],
\R^d\setminus[t_{k},t_{j-1}] \bigr](x)
\]
and
\[
u_{k,j}(x):=T \bigl[[t_{k-1},t_{j}],
\R^d\setminus[t_{k-2},t_{j+1}] \bigr](x),
\]
where we have used the convention that $[s,t]=\varnothing$ if $s\nleq t$
and that the addition of an integer to a multi-index is the addition of
the integer to every component of the multi-index.

For each $k\le j$, we have
\begin{eqnarray*}
\|l_{k,j}-u_{k,j}\|_s^s &\le& \mu
\bigl([t_{k-2},t_{j+1}]\setminus[t_{k+1},t_{j-2}]
\bigr)
\\
&\le& \sum_{i=1}^d
\bigl(\bigl|F_i(t_{i,k_i+1})-F_i(t_{i,k_i-2})\bigr|+\bigl
|F_i(t_{i,j_i+1})-F_i(t_{i,j_i-2})\bigr|\bigr)
\\
&\le&2\frac{3d}{m}
\end{eqnarray*}
and thus $\|l_{k,j}-u_{k,j}\|_s\le\varepsilon$.
Moreover, since for $a< b<b' < a'$,
\[
d\bigl(\bigl[b,b'\bigr],\R^d\setminus
\bigl[a,a'\bigr]\bigr)=\min_{i=1,\ldots,d} \min\bigl\{
|a_i-b_i|,\bigl|a'_i-b'_i\bigr|
\bigr\}
\]
using Lemma \ref{lemtransition} and (\ref{eq-modulus}), we have
\begin{eqnarray*}
\|l_{k,j}\|_\alpha&\le& 1 + 3^\alpha\Bigl(\min
_{i=1,\ldots,d} \min\bigl\{|t_{i,k_i}-t_{i,k_i+1}|,|t_{i,j_i-1}-t_{i,j_i-2}|
\bigr\} \Bigr)^{-\alpha}
\\
&\le& 1 + 3^\alpha\biggl[\inf\biggl\{x>0\dvt \exists i\in\{
1,\ldots,d\},
\exists t, F_i(t+x)-F_i(t)\ge\frac{1}{m} \biggr\}
\biggr]^{-\alpha}
\\
&\le& 1+ 3^\alpha\biggl[\inf\biggl\{x>0\dvt c_F\bigl|\log
(x)\bigr|^{-s\gamma}
\ge\frac{1}{m} \biggr\} \biggr]^{-\alpha}
\\
&\le& 1 + 3^\alpha\exp\bigl( \alpha(c_F m)^{{1}/({s\gamma
})}
\bigr),
\end{eqnarray*}
where $c_F$ is given by (\ref{eq-modulus}). The same bound holds for
$\|u_{k,j}\|_\alpha$.

Thus, there exists a new constant $C_F>0$ such that for all $k\le j \in
\{0,\ldots,m\}^d$, $[l_{k,j},u_{k,j}]$ is an $(\varepsilon,\exp
(C_F\varepsilon^{-{1}/{\gamma}}),\G,L^s(\mu))$-bracket.
It is clear that for each function $f\in\F$ there exists a bracket of
the form $[l_{k,j},u_{k,j}]$ which contains $f$.
Further, we have at most $(m+1)^{2d}$ such brackets, which proves the
proposition.
\end{pf}

Notice that under the assumptions of the proposition, condition (\ref{bracket})
is satisfied and therefore Theorem \ref{thep-clt} may be applied to
empirical processes
indexed by the class of indicators of rectangles, taking $\B$ to be
the class of
bounded H\"older functions.

%%%%%%%%%%%%%%%%%%%%%%%%%%%%%%%%%%%%%%%%%%%%%%%%%%%%%%%%%%%%%%%%%%%%%%%%%%%%%%%%%%%%%%%%%%%%%%%%%%%%%%%%%%%%%%%%%%%%%%%%%%%%%%%%%

%co3.3 #&#
%
\begin{corollary}
Let $(X_i)_{i\ge0}$ be an $\R^d$-valued stationary process. Let $\F$
be the
class of indicator functions of rectangles in $\R^d$ and let $\G=\H
_\alpha(\R^d,[0,1])$.
Assume that, for some $s\ge1$, $a\in\R$, and $\gamma>\max\{2+a,1\}
$, Assumptions \ref{assump1} and \ref{assump2} hold, and that the distribution
function of the $X_i$ satisfies (\ref{eq-modulus}). Then the empirical
process $(U_n(f))_{f\in\F}$ converges in distribution in $\ell
^\infty(\F)$
to a tight Gaussian process.
\end{corollary}

%re3.4 #&#
%
\begin{remark}
By regarding the class of indicator functions of left infinite rectangles
as a sub-class of $\F$,
we obtain Theorem 1 of Dehling and Durieu \cite{DehDur11}
as a particular case of the preceding corollary.
\end{remark}

%%%%%%%%%%%%%%%%%%%%%%%%%%%%%%%%%%%%%%%%%%%%%%%%%%%%%%%%%%%%%%%%%%%%%%%%%%%%%%%%%%%%%%%%%%%%%%%%%

%s3.2 #&#
\subsection{Example 2: Indicators of multidimensional balls in the
unit cube}\label{exaballs}

Here, we consider the class $\F$ of indicator functions of balls on
$\mathcal{X}=[0,1]^d$, that is,
\[
\F:=\bigl\{1_{\ball(x,r)}\dvt x\in[0,1]^d, r\ge0\bigr\},
\]
where $\ball(x,r)=\{y\in[0,1]^d\dvt |x-y|<r\}$.
We have the following upper bound.
%
%pr3.5 #&#
%
\begin{proposition}\label{prop-balls}
Let $\mu$ be a probability distribution on $[0,1]^d$ with a density
bounded by some $B>0$ and let $s\ge1$. Then there exists a constant
$C=C_{d,B}>0$ such that
\[
N\bigl(\varepsilon,C\varepsilon^{-\alpha s},\F,\G,L^s(\mu)
\bigr)=\RMO_{d,B} \bigl({\varepsilon^{-(d+1)s}} \bigr) \qquad
\mbox{as }
\varepsilon\to0,
\]
where $\G=\H_\alpha([0,1]^d,[0,1])$.
\end{proposition}

Note that the second argument in the bracketing number is different
from the one appearing in the condition (\ref{bracket}).
In this situation, we have a stronger type of bracketing number than in
(\ref{bracket}).

\begin{pf*}{Proof of Proposition \ref{prop-balls}}
Let $\varepsilon>0$\vspace*{1pt} be fixed and $m=\lfloor\varepsilon^{-s}\rfloor$.
For all $i=(i_1,\ldots,i_d) \in\{0,\ldots,m\}^d$, we denote by $c_i$
the center of the rectangle
$[\frac{i_1-1}{m},\frac{i_1}{m}]\times\cdots\times[\frac
{i_d-1}{m},\frac{i_d}{m}]$.
Then we define, for $i\in\{1,\ldots,m\}^d$ and $j\in\{0,\ldots,m\}$,
the functions
\[
l_{i,j}(x):= T \biggl[\ball\biggl(c_{i},
\frac{j-2}{m}\sqrt{d} \biggr),[0,1]^d\Big\backslash\ball
\biggl(c_{i},\frac{j-1}{m}\sqrt{d} \biggr) \biggr](x)
\]
and
\[
u_{i,j}(x):=T \biggl[\ball\biggl(c_{i},\frac{j+2}{m}
\sqrt{d} \biggr), [0,1]^d\Big\backslash\ball\biggl(c_{i},
\frac{j+3}{m}\sqrt{d} \biggr) \biggr](x),
\]
where we use the convention that a ball with negative radius is the
empty set.

By Lemma \ref{lemtransition}, these functions are $\alpha$-H\"older
and, since $d(\ball(x,r),\R^d\setminus\ball(x,r'))= r'-r$, we have
\[
\|l_{i,j}\|_\alpha\le1 + \biggl(\frac{3m}{\sqrt{d}}
\biggr)^\alpha\le1+3\varepsilon^{-s\alpha}.
\]
The same bound holds for $\|u_{i,j}\|_\alpha$.
Since $\mu$ has a bounded density with respect to Lebesgue measure, we
also have
\begin{eqnarray*}
\|l_{i,j}-u_{i,j}\|_s^s&\le& \mu
\biggl(\ball\biggl(c_{i},\frac
{j+3}{m}\sqrt{d} \biggr)\Big\backslash
\ball\biggl(c_{i},\frac
{j-2}{m}\sqrt{d} \biggr) \biggr)
\\
&\le& B c_d \biggl( \biggl(\frac{j+3}{m}\sqrt{d}
\biggr)^d- \biggl(\frac{j-2}{m}\sqrt{d} \biggr)^d
\biggr),
\end{eqnarray*}
where $c_d$ is the constant $\frac{\pi^{d/2}}{\Gamma(d/2+1)}$
($\Gamma$ is the gamma function).
Hence,
\[
\|l_{i,j}-u_{i,j}\|_s\le c_{d,B}^{1/s}
\varepsilon
\]
as $\varepsilon\to0$, where $c_{d,B}$ is a constant depending only on
$d$ and $B$.\vadjust{\eject}

Now, if $f$ belongs to $\F$, then $f=1_{\ball(x,r)}$ for some $x\in
[0,1]^d$, and $0\le r \le\sqrt{d}$.
Thus, there exist some $i=(i_1,\ldots,i_d)\in\{0,\ldots,m\}^d$ and
$j\in\{0,\ldots,m\}$ such that
\[
x\in\biggl[\frac{i_1-1}{m},\frac{i_1}{m} \biggr)\times\cdots
\times\biggl[
\frac{i_d-1}{m},\frac{i_d}{m} \biggr) \quad\mbox{and}\quad\frac
{j}{m}\sqrt{d}\le
r\le\frac{j+1}{m}\sqrt{d}.
\]
We then have $l_{i,j}\le f\le u_{i,j}$.

Thus, the $(m+1)m^{d}$ brackets $[l_{i,j},u_{i,j}]$,
$i\in\{1,\ldots,m\}^d$ and $j\in\{0,\ldots,m\}$, cover the
class~$\F$.
Therefore, $N(c_{d,B}^{1/s}\varepsilon,4\varepsilon^{-\alpha
s},\F,\G,L^s(\mu))=\RMO_{d,B}(\varepsilon^{-(d+1)s})$ as $\varepsilon
\to
0$, which implies that there exists a constant $C_{d,B}>0$, for which
$N(\varepsilon,C_{d,B}\varepsilon^{-\alpha s},\F,\G,L^s(\mu
))=\RMO_{d,B}(\varepsilon^{-(d+1)s})$ as \mbox{$\varepsilon\to0$}.
\end{pf*}

%%%%%%%%%%%%%%%%%%%%%%%%%%%%%%%%%%%%%%%%%%%%%%%%%%%%%%%%%%%%%%%%%%%%%%%%%%%%%%%%%%%%%%%%%%%%%%%%%%%%%%%%%%%%%%%%%%%%%%%%%%%%%%%

%s3.3 #&#
\subsection{Example 3: Indicators of uniformly bounded
multidimensional ellipsoids centered in the unit cube}
\label{exaellipsoidscube}
Set $\mathcal{X}=\R^d$.
Here, we consider the class of ellipsoids which are aligned with the
coordinate axes, have their center in $[0,1]^d$, and their parameters
bounded by some constant $D>0$.
Without loss of generality, we assume that $D\in\N$.
For $x=(x_1,\ldots,x_d)\in[0,1]^d$ and all $r=(r_1,\ldots,r_d)\in
[0,D]^d$, we set
\[
\ellip(x,r):= \Biggl\{y\in\R^d\dvt \sum_{i=1}^d
\frac
{(y_i-x_i)^2}{r_i^2}\le1 \Biggr\}.
\]
We denote by $\F$ the class of indicator functions of these
ellipsoids, that is,
\[
\F:=\bigl\{1_{\ellip(x,r)}\dvt x\in[0,1]^d, r\in[0,D]^d
\bigr\}.
\]
We have the following upper bound.

%pr3.6 #&#
%
\begin{proposition}\label{proellipsoidsunitsquare}
Let $\mu$ be a probability distribution on $\R^d$ with a density
bounded by some $B>0$. Then there exists a constant $C=C_{d,B,D}>0$
such that
\[
N\bigl(\varepsilon,C\varepsilon^{-2\alpha s},\F,\G,L^s(\mu)
\bigr)=\RMO_{d,B} \bigl({\varepsilon^{-2ds}} \bigr) \qquad\mbox{as }
\varepsilon\to0,
\]
where $\G=\H_\alpha(\R^d,[0,1])$.
\end{proposition}

\begin{pf}
Let $\varepsilon>0$ be fixed and $m=\lfloor\varepsilon^{-s}\rfloor$.
For all $i=(i_1,\ldots,i_d) \in\{0,\ldots,m\}^d$, we denote by $I_i$ the
rectangle $[\frac{i_1-1}{m},\frac{i_1}{m}]\times\cdots\times
[\frac{i_d-1}{m},\frac{i_d}{m}]$. Then, for $i\in\{1,\ldots,m\}^d$ and
$j=(j_1,\ldots,j_d)\in\{0,\ldots,Dm-1\}^d$, we define the sets
\[
U_{i,j}=\bigcup_{x\in I_i} \ellip\biggl(x,
\frac{j}{m} \biggr)= \Biggl\{y\in\R^d\dvt \min
_{x\in I_i}\sum_{k=1}^d
\frac
{(y_k-x_k)^2}{j_k^2}\le\frac{1}{m^2} \Biggr\}
\]
and
\[
L_{i,j}=\bigcap_{x\in I_i} \ellip\biggl(x,
\frac{j}{m} \biggr)= \Biggl\{y\in\R^d\dvt \max
_{x\in I_i}\sum_{k=1}^d
\frac
{(y_k-x_k)^2}{j_k^2}\le\frac{1}{m^2} \Biggr\}.
\]
We introduce the bracket $[l_{i,j},u_{i,j}]$ given by
\[
l_{i,j}(x):= T \bigl[ L_{i,j-1}, \R^d\setminus
L_{i,j} \bigr](x) \quad\mbox{and}\quad u_{i,j}(x):=T
\bigl[U_{i,j+1}, \R^d\setminus U_{i,j+2} \bigr](x),
\]
where we use the convention that an ellipsoid with one negative
parameter is the empty set.
By Lemma \ref{lemtransition}, these functions are $\alpha$-H\"older.
Further, we have the following lemma which is proved in the \hyperref
[app]{Appendix}.
%
%le3.7 #&#
%
\begin{lemma}\label{lemdistellip}
For all $j\in\{0,\ldots,Dm-1\}^d$, $x\in\R^d$, we have
\[
d \biggl(\ellip\biggl(x,\frac{j}{m} \biggr), \R^d\Big\backslash
\ellip\biggl(x,\frac{j+1}{m} \biggr) \biggr) \geq D^{-1}m^{-2}.
\]
\end{lemma}
As a consequence, we infer that
the distance between $U_{i,j}$ and $\R^d\setminus U_{i,j+1}$ is at
least $D^{-1}m^{-2}$ and the distance between $L_{i,j}$ and $\R
^d\setminus L_{i,j+1}$ is at least $D^{-1}m^{-2}$.
Thus, by Lemma \ref{lemtransition},
we have
\[
\|l_{i,j}\|_\alpha\le1 + 3^\alpha D^{\alpha}m^{2\alpha}
\le1+ 3D\varepsilon^{-2\alpha s},
\]
and the same bound holds for $\|u_{i,j}\|_\alpha$.

%f1 #&#
%
\begin{figure}

\includegraphics{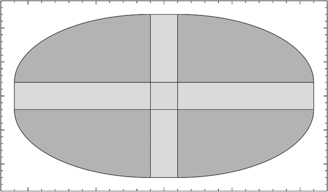}

\caption{$U_{i,j}$ in dimension 2.}
\label{ellip1}
\end{figure}

Now, to bound $\|u_{i,j}-l_{i,j}\|_s$ we need to estimate the Lebesgue
measures of $U_{i,j}$ and $L_{i,j}$.
Recall that, if $j=(j_1,\ldots,j_d)\in\R_+^d$ and $x\in\R^d$, the
Lebesgue measure of the ellipsoid $\ellip(x,j)$ is given by
\[
\lambda\bigl(\ellip(x,j)\bigr)=c_d \prod
_{k=1}^d j_k,
\]
where $c_d$ is the constant $\frac{\pi^{d/2}}{\Gamma(d/2+1)}$.
The set $U_{i,j}$ can be seen as the set constructed as follows: start
from an ellipsoid of parameters $j/m$ centered at the center of $I_i$,
cut it along its hyperplanes of symmetry, and shift each obtained
component away from the center by a distance of $1/2m$ in every
direction; $U_{i,j}$ is then the convex hull of these $2^d$ components
(see Figure \ref{ellip1} for the dimension 2).
Let us denote by $V_{i,j}$ the set that has been added to the $2^d$
components to obtain the convex hull.
We can bound the volume of $U_{i,j}$ by the volume of the ellipsoid
plus a bound on the volume of $V_{i,j}$, that is,
\[
\lambda(U_{i,j}) \le c_d\prod
_{k=1}^d\frac{j_k}{m}+\sum
_{k=1}^d\frac{1}{m}\prod
_{l\ne k}\frac{2j_l+1}{m}.
\]
The set $L_{i,j}$ can be seen as the intersection of the $2^d$
ellipsoids of parameters $j/m$ centered at each corner of the hypercube
$I_i$ (see Figure \ref{ellip2} for the dimension 2). Its volume is
larger than the volume of an ellipsoid of parameters $j/m$ minus the
volume of $V_{i,j}$. We thus have
\[
\lambda(L_{i,j}) \ge c_d\prod
_{k=1}^d\frac{j_k}{m}-\sum
_{k=1}^d\frac{1}{m}\prod
_{l\ne k}\frac{2j_l+1}{m}.
\]
Since $\mu$ has a bounded density with respect to Lebesgue measure, we have
\begin{eqnarray*}
\|l_{i,j}-u_{i,j}\|_s^s&\le& \mu
(U_{i,j+2}\setminus L_{i,j-1} )
\\
&\le& B \lambda(U_{i,j+2})- B \lambda(L_{i,j-1}).
\end{eqnarray*}
We infer $\|l_{i,j}-u_{i,j}\|_s=c_{d,B}^{1/s}(\varepsilon)$, as
$\varepsilon\to0$, where the constant $c_{d,B}$ only depends on $d$
and~$B$.

%f2 #&#
%
\begin{figure}

\includegraphics{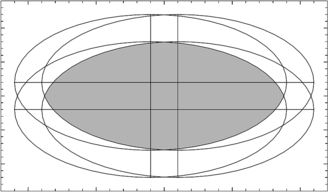}

\caption{$L_{i,j}$ in dimension 2.}
\label{ellip2}
\end{figure}

Now, if $f$ belongs to $\F$, then $f=1_{\ellip(x,r)}$ for some $x\in
\mathcal{X}$, and $r\in[0,D]^d$.
Thus, there exist some $i=(i_1,\ldots,i_d)\in\{0,\ldots,m\}^d$ and
$j\in\{0,\ldots,Dm-1\}^d$ such that
\[
x\in\biggl[\frac{i_1-1}{m},\frac{i_1}{m} \biggr)\times\cdots
\times\biggl[
\frac{i_d-1}{m},\frac{i_d}{m} \biggr)
\]
and for each $k=1,\ldots,d$,
\[
\frac{j_k}{m}\le r_k\le\frac{j_k+1}{m}.
\]
We then have $l_{i,j}\le f\le u_{i,j}$.

Thus, the $D^dm^{2d}$ brackets $[l_{i,j},u_{i,j}]$, $i\in\{1,\ldots
,m\}
^d$ and $j\in\{0,\ldots,Dm-1\}^d$, cover the class $\F$. Therefore,
there exists a $C_{d,B,D}>0$, such that
$N(\varepsilon,C_{d,B,D}\varepsilon^{-\alpha s},\F,\G,L^s(\mu
))=\RMO_{d,B}(\varepsilon^{-2ds})$, as $\varepsilon\to0$.
\end{pf}

%%%%%%%%%%%%%%%%%%%%%%%%%%%%%%%%%%%%%%%%%%%%%%%%%%%%%%%%%%%%%%%%%%%%%%%%%%%%%%%%%%%%%%%%%%%%%%%%%%%%%%%%%%%%%%%%%%%%%%%%%%%%%%%

%s3.4 #&#
\subsection{Example 4: Indicators of uniformly bounded
multidimensional ellipsoids}
In Example \hyperref[exaellipsoidscube]{3}, we only considered indicators of
ellipsoids centered in a compact subset of $\R^d$, namely the unit
square. The following lemma will allow us to extend such results to
indicators of sets in the whole $\R^d$, at the cost of a moderate
additional assumption and a marginal increase of the bracketing numbers.

%le3.8 #&#
%
\begin{lemma}\label{lemextension-lemma}
Let $\mu$ be a measure with continuous distribution function $F$,
and $s\ge1$. Furthermore let $\F:=\{1_{S}\dvt S \in\mathcal{S}\}$,
where $\mathcal{S}$ is a class of measurable sets of diameter not
larger than $D\ge1$, and $\G=\H_\alpha(\R^d,[0,1])$.
Assume that there are constants $p, q \in\N$, $C>0$, and a function
$f\dvtx\R_+\rightarrow\R_+$, such that for any $K>0$ we have
%
%e3.2 #&#
%
\begin{equation}\label{eqtemp100}
N \bigl(\varepsilon,f(\varepsilon),\F_{K},\G,L^s(\mu)
\bigr) \leq C K^p \varepsilon^{-q}
\end{equation}
for sufficiently small $\varepsilon$, where $\F_K:=\{1_{S}\dvt S\in
\mathcal{S}, S\subset[-K,K]^d \}$.
If there are some constants $b, \beta>0$ such that
%
%e3.3 #&#
%
\begin{equation}\label{eqtmpa03}
\mu\bigl( \bigl\{ x\in\R^d\dvt |x| > t \bigr\} \bigr) \le b
t^{-{1}/{\beta}}
\end{equation}
for all sufficiently large $t$,
then
\begin{eqnarray*}
&&
N \bigl(\varepsilon,\max\bigl\{f(\epsilon),4\sqrt{d}\bigl(\omega
_F^{-1}\bigl(2^{-(d+1)}\varepsilon^s
\bigr)\bigr)^{-\alpha} \bigr\},\F,\G,L^s(\mu) \bigr) \\
&&\quad=
\RMO_{\beta,b,C,D,p}\bigl(\varepsilon^{-(\beta p s + q)}\bigr) \qquad
\mbox{as
}\varepsilon
\rightarrow0,
\end{eqnarray*}
where $\omega_F$ is the modulus of continuity of $F$.
\end{lemma}

The proof is postponed to the \hyperref[app]{Appendix}.

%pr3.9 #&#
%
\begin{proposition}\label{corellipsarb}
Let $\mathcal{F}$ denote the class of indicators of ellipsoids of
diameter uniformly bounded by $D>0$, which are aligned with coordinate
axes (and arbitrary centers in the whole space $\R^d$).
If $\mu$ is a measure on $\R^d$ with a density bounded by $B>0$ and
if furthermore (\ref{eqtmpa03}) holds for some $\beta>0$ and $b>0$,
then there exists a constant $C=C_{d,B,D}>0$ such that
\[
N \bigl(\varepsilon,C \varepsilon^{-2\alpha s},\F,\G,L^s(\mu)
\bigr)=\RMO_{\beta,b,d,B,D,s} \bigl(\varepsilon^{-(\beta
s+2)ds} \bigr) \qquad\mbox{as }
\varepsilon\rightarrow0,
\]
where $\G=\H_\alpha(\R^d,[0,1])$.
\end{proposition}

\begin{pf}
In the situation of Example \hyperref[exaellipsoidscube]{3} change the set of
the centers of the ellipsoids $[0,1]^d$ to $[-K,K]^d$ and apply Lemma
\ref{lemextension-lemma}.
Following the proof of Proposition \ref{proellipsoidsunitsquare},
we can easily see that condition (\ref{eqtemp100}) holds for $p=ds$,
$q=2ds$ and $f(\varepsilon)=C_{d,B,D}\varepsilon^{-2\alpha s}$.
Note that since we have a bounded density, we have $\omega_F(x) \leq B
x$ and therefore $4\sqrt{d}(\omega_F^{-1}(2^{-(d+1)}\varepsilon
^s))^{-\alpha} \leq4\sqrt{d} (2^{d+1}B)^{\alpha} \varepsilon
^{-\alpha s}\le C_{d,B,D} \varepsilon^{-2\alpha s}$ for sufficiently
small $\varepsilon$.
\end{pf}

%re3.10 #&#
%
\begin{remark}
In the situation of Proposition \ref{corellipsarb}
for the class $\F'$ of indicators of balls in $\R^d$ with uniformly
bounded diameter,
we can obtain the slightly sharper bound
\[
N\bigl(\varepsilon,C\varepsilon^{-\alpha s},\F',
\G,L^s(\mu)\bigr)=\RMO_{\beta,b,d,B,D,s}\bigl({\varepsilon^{-((\beta
+1)ds+1)s}}
\bigr) \qquad\mbox{as }\varepsilon\rightarrow0
\]
for some $C=C'_{d,B}>0$ by applying Lemma \ref{lemextension-lemma}
directly to the situation in Example \hyperref[exaballs]{2} and using the same
arguments as in the previous example.
\end{remark}

%%%%%%%%%%%%%%%%%%%%%%%%%%%%%%%%%%%%%%%%%%%%%%%%%%%%%%%%%%%%%%%%%%%%%%%%%%%%%%%%%%%%%%%%%%%%%%%%%%%%%%%%%%%%%%%%%%%%%%%%%%%%%%%

%s3.5 #&#
\subsection{Example 5: Indicators of balls of an arbitrary metric with
common center}\label{exazero-centered-balls}
Let $(\X,d)$ be a metric space and fix $x_0\in\X$.
An $x_0$-centered ball is given by
\[
\ball(t):= \bigl\{x\in\X\dvt d(x_0,x) \leq t\bigr\}.
\]
We have the following bound on the bracketing numbers of the class
$\mathcal{F}:=\{1_{\ball(t)}\dvt t>0 \}$.
%
%pr3.11 #&#
%
\begin{proposition}\label{prozero-centered}
Let $s\ge1$ and $\gamma>1$. If for the probability measure $\mu$ on
$\X$ the modulus of continuity $\omega_G$ of the function
$G(t):= \mu(\ball(t))$ satisfies
%
%e3.4 #&#
%
\begin{equation}\label{eqmodcontG}
\omega_G(x) = \RMO\bigl(|{\log x}|^{-s\gamma}\bigr) \qquad\mbox{as }x
\to0,
\end{equation}
then there is a constant $C=C_{G}>0$ such that
\[
N\bigl(\varepsilon,\exp\bigl(C\varepsilon^{-{1}/{\gamma}}\bigr
),\F,\G,L^s(\mu)\bigr)
= \RMO\bigl(\varepsilon^{-s}\bigr)\qquad\mbox{as } \varepsilon\to0,
\]
where $\G=\H_\alpha(\X,[0,1])$.
\end{proposition}

%re3.12 #&#
%
\begin{remark}
Note that in the case that $\mathcal{X}=\R^2$, $\mrmdd\mu(t)=\rho
(t)\mrmd t$,
the metric $d$ is given by the Euclidean norm, and $x_0=0$, an
equivalent condition to (\ref{eqmodcontG}) is
\[
\sup_{r\geq0} \int_r^{r+x} t \int
_0^{2\uppi} \rho\bigl(t \mrme^{\mrmi\varphi}\bigr) \mrmd
\varphi\mrmd t = \RMO\bigl(|{\log x}|^{-s\gamma}\bigr)\qquad\mbox
{as }x\to0.
\]
\end{remark}

\begin{pf*}{Proof of Proposition \ref{prozero-centered}}
Fix $\varepsilon>0$ and choose $m=\lfloor3\varepsilon^{-s} +1\rfloor
$. Let $G^{-1}$ denote the pseudo-inverse of $G$
and set for $i\in\{1,\ldots,m\}$
\[
r_i:= G^{-1} \biggl( \frac{i}{m} \biggr),\qquad
\ball_i:=\ball(r_i).
\]
For convenience, set $\ball_{-1}, \ball_{0}:= \varnothing$ and
$\ball
_{m+1}=\X$.
Define
\[
l_i(x):=T [\ball_{i-2},\X\setminus\ball_{i-1}
](x) \quad\mbox{and}\quad u_i(x):=T [\ball_{i},\X\setminus
\ball_{i+1} ](x).
\]
The system $\{[l_i, u_i]\dvt i\in\{1,\ldots,m\}\}$ is a covering for
$\mathcal{\F}$. Obviously
\[
\| u_i -l_i \|_s^s \le
\mu(B_{i+1}\setminus B_{i-2})\le\frac3m \le
\varepsilon^s.
\]
By Lemma \ref{lemtransition}, we have
\[
\|u_i\|_{\alpha} \le1 + \frac{3^\alpha}{d(\ball_i,\X\setminus
\ball_{i+1})^{\alpha}} \leq1 +
\frac{3^\alpha}{(r_{i+1}-r_i)^{\alpha}}.
\]
Since by condition (\ref{eqmodcontG})
\begin{eqnarray*}
r_{i+1}-r_i &\geq& \inf\biggl\{ x> 0\dvt \exists t\in\R
\mbox{ such that } G(t+x)-G(t) \geq\frac{1}{m} \biggr\}
\\
&\geq&\inf\biggl\{ x> 0\dvt \exists t\in\R\mbox{ such that }
\omega
_{G}(x) \geq\frac{1}{m} \biggr\}
\\
&\geq&\exp\bigl(-c_G m^{{1}/({s\gamma})}\bigr)
\end{eqnarray*}
for some constant $c_G>0$, there is a constant $C_G>0$ such that
\begin{eqnarray*}
\|u_i\|_{\alpha} &\leq&1 + 3^\alpha\exp\bigl(\alpha
c_G m^{{1}/({s\gamma})}\bigr) \leq\exp\bigl(C_{G}
m^{{1}/({s\gamma})}\bigr) \\
&\leq&\exp\bigl(C_{G} \varepsilon^{-
{1}/{\gamma}}
\bigr).
\end{eqnarray*}
Analogously, we can show that $\|l_i\|_{\alpha} \leq\exp(C_{G}
\varepsilon^{-{1}/{\gamma}})$. This implies that all $[l_i,u_i]$
are
$(\varepsilon,\exp(C_{G}\varepsilon^{-{1}/{\gamma}}),\F,\G
,L^s(\mu))$-brackets
and thus the proposition is proved.
\end{pf*}

%s3.6 #&#
\subsection{Example 6: A class of monotone functions}\label{exaclass-funct}
In this example, we choose $\X=\R$. We consider the case of a
one-parameter class of functions $\F=\{f_t\dvt t\in[0,1]\}$, where
$f_t$ are functions from $\R$ to $\R$ with the properties:

\begin{longlist}
\item for all $t\in[0,1]$ and $x\in\R$, $0\le f_t(x) \le1$;
\item for all $0\le s\le t\le1$, $f_s\le f_t$;
\item for all $t\in[0,1]$, $f_t$ is non-decreasing on $\R$.
\end{longlist}
Note that all the sequel remains true if in (iii),
non-decreasing is replaced by non-increasing.
Further, for a probability measure $\mu$ on $\R$, we define $G_\mu
(t)=\mu f_t$ and we say that $G_\mu$ is Lipschitz with Lipschitz
constant $\lambda>0$ if $ |G_\mu(t)-G_\mu(s)|\le\lambda|t-s|$,
for all $s,t\in[0,1]$.

Empirical processes indexed by a 1-parameter class of functions arise,
for example, in the study of empirical U-processes; see Borovkova,
Burton and Dehling \cite{BorBurDeh01}. The empirical \mbox{U-distribution}
function with kernel
function $g(x,y)$ is defined as
\[
U_n(t)=\frac{1}{{n\choose2}} \sum_{1\leq i<j\leq n}
1_{\{
g(X_i,X_j)\leq t \}}.
\]
Then, the first order term in the Hoeffding decomposition is given by
\[
\sum_{i=1}^n g_{t}(X_i),
\]
where $g_{t}(x)=P(g(x,X_1)\leq t)$. For this class of functions,
conditions (i) and (ii) are automatically satisfied.
Condition (iii) holds, if $g(x,y)$ is monotone in $x$. This is,
for example, the case for the kernel
$g(x,y)=y-x$, which arises in the study of the empirical correlation
integral; see Borovkova, Burton and Dehling \cite{BorBurDeh01}.

%pr3.13 #&#
%
\begin{proposition}
Let $s\ge1$ and $\gamma>1$. Let $\mu$ be a probability measure on
$\R$ such that its distribution function $F$ satisfies
%
%e3.5 #&#
%
\begin{equation}
\label{modulus} \omega_F(x)=\RMO\bigl(\bigl|\log(x)\bigr
|^{-s\gamma}\bigr)
\qquad\mbox{as }x\to0\vadjust{\goodbreak}
\end{equation}
and such that $G_\mu$ is Lipschitz with Lipschitz constant $\lambda>0$.
Then there exists a $C=C_F>0$, such that
\[
N\bigl(\varepsilon,\exp\bigl( C
\varepsilon^{-{1}/{\gamma}}\bigr),\F,\G,L^s(\mu)\bigr)
=\RMO_{\lambda} \bigl(\varepsilon^{-s} \bigr) \qquad\mbox{as
}\varepsilon\rightarrow0,
\]
where $\G=\H_\alpha(\R,[0,1])$.
\end{proposition}

\begin{pf}
Let $\varepsilon>0$ and $m=\lfloor(\lambda+4)\varepsilon^{-s}
+1\rfloor$.
For $i=0,\ldots,m$, we set
\[
t_i=\frac im \quad\mbox{and}\quad x_i=F^{-1} \biggl(
\frac im \biggr).
\]
We always have $x_m=+\infty$, but $x_0$ could be finite or $-\infty$.
In order to simplify the notation, in the first case, we change to
$x_0=-\infty$.

We define, for $j\in\{1,\ldots,m\}$, the functions $l_j$ and $u_j$ as follows.
If $k\in\{1,\ldots,m-1\}$, we set $l_j(x_k)=f_{t_{j-1}}(x_{k-1})$ and
$u_j(x_k)=f_{t_j}(x_{k+1})$, where we have to understand $f(\pm\infty
)$ as $\lim_{x\to\pm\infty}f(x)$.
If $k\in\{0,\ldots,m-1\}$ and $x\in(x_k,x_{k+1})$, we define
$l_j(x)$ and $u_j(x)$ by the linear interpolations,
\begin{eqnarray*}
l_j(x)&=&l_j(x_k)+(x-x_k)
\frac{l_j(x_{k+1})-l_j(x_k)}{x_{k+1}-x_k},
\\
u_j(x)&=&u_j(x_k)+(x-x_k)
\frac{u_j(x_{k+1})-u_j(x_k)}{x_{k+1}-x_k}
\end{eqnarray*}
with the exceptions that $l_j(x)=l_j(x_1)=f_{t_{j-1}}(-\infty)$ if
$x\in(-\infty,x_1)$ and $u_j(x)=u_j(x_{m-1})=f_{t_j}(+\infty)$ if
$x\in(x_{m-1},+\infty)$.
Then it is clear that for all $t_{j-1}\le t\le t_j$, we have $l_j\le
f_t \le u_j$, that is, $f_t$ belongs to the bracket $[l_j,u_j]$.

Further, being piecewise affine functions, $l_j$ and $u_j$ are $\alpha
$-H\"older continuous functions with H\"older norm
\[
\|l_j\|_\alpha\le1 + \max_{k=1,\ldots,m} \frac
{l_j(x_k)-l_j(x_{k-1})}{(x_{k}-x_{k-1})^\alpha}\le1 +
\max_{k=1,\ldots,m} \frac{1}{(x_{k}-x_{k-1})^\alpha} \le1+ \exp
\bigl(
C_F m^{{1}/({s\gamma})} \bigr).
\]
Here we have used the condition (\ref{modulus}) and the same
computation as for the class of indicators of rectangles.
Analogously, the same bound holds for $\|u_j\|_\alpha$.

Now,
\[
\|u_j-l_j\|_s^s\le
\|u_j-l_j\|_1\le\|u_j-f_{t_j}
\|_1 + \| f_{t_j}-f_{t_{j-1}}\|_1 +
\|l_j-f_{t_{j-1}}\|_1.
\]
First, since $G_\mu$ is Lipschitz, we have
\[
\|f_{t_j}-f_{t_{j-1}}\|_1\le G(t_j)-G(t_{j-1})
\le\lambda(t_j-t_{j-1}) = \frac{\lambda}{m}.
\]
For $x\in[x_{k-1},x_k]$, since $f_t$ is non-decreasing, we have
$u_j(x)\le f_{t_j}(x_{k+1})$ and $u_{t_j}(x)\ge f_{t_j}(x_{k-1})$, thus
\begin{eqnarray*}
\|u_j-f_{t_j}\|_1 &\le& \sum
_{k=1}^{m-1} \bigl|f_{t_j}(x_{k+1})-f_{t_j}(x_{k-1})\bigr|
\mu\bigl([x_k,x_{k+1}]\bigr)
\\
&\le& \frac1m \sum_{k=1}^{m-1}
\bigl(\bigl|f_{t_j}(x_{k+1})-f_{t_j}(x_{k})\bigr|+\bigl
|f_{t_j}(x_{k})-f_{t_j}(x_{k-1})\bigr|
\bigr)
\\
&\le& \frac2m \sum_{k=0}^{m-1}
\bigl|f_{t_j}(x_{k+1})-f_{t_j}(x_{k})\bigr|
\\
&\le& \frac2m
\end{eqnarray*}
since, by monotonicity, $\sum_{k=0}^{m-1}
|f_{t_j}(x_{k+1})-f_{t_j}(x_{k})|\le1$.
In the same way, we get $\|l_j-f_{t_{j-1}}\|_1\le\frac2m$ and we infer
\[
\|u_j-l_j\|_s\le\biggl(\frac{\lambda+4}{m}
\biggr)^{1/s}\le\varepsilon.
\]
Thus, the number of $(\varepsilon,\exp( C_F \varepsilon^{-
{1}/{\gamma}}),\G,L^s(\mu))$-brackets needed to cover the class $\F$
is bounded by $m$, which proves the proposition.
\end{pf}

%%%%%%%%%%%%%%%%%%%%%%%%%%%%%%%%%%%%%%%%%%%%%%%%%%%%%%%%%%%%%%%%%%%%%

%s4 #&#
\section{Application to ergodic torus automorphisms}\label{secappli}

We can apply Theorem \ref{thep-clt} to the empirical process of
ergodic torus automorphisms. Let $\mathbb{T}^d=\R^d/\Z^d$ be the
torus of
dimension $d>1$, which is identified with $[0,1]^d$. If $A$ is a square
matrix of dimension $d$ with integer
coefficients and determinant $\pm1$, then the transformation
$T\dvtx\mathbb{T}^d\longrightarrow\mathbb{T}^d$ defined by
\[
Tx=Ax\mbox{ mod }1
\]
is an automorphism of $\mathbb{T}^d$ that preserves the Lebesgue
measure $\lambda$. Thus $(\mathbb{T}^d, \mathcal{B}(\mathbb{T}^d),
\lambda, T)$ is a measure preserving dynamical system. It is ergodic if
and only if the matrix $A$ has no eigenvalue which is a root of unity.
A result of Kronecker shows that in this case, $A$ always has at least
one eigenvalue which has modulus different than $1$. The hyperbolic
automorphisms (i.e., no eigenvalue of modulus $1$) are particular cases
of Anosov diffeomorphisms. Their properties are better understood than
in the general case. However, the general case of ergodic automorphisms
is an example of a partially hyperbolic system for which strong results
can be proved. The central limit theorem for regular observables has
been proved by Leonov \cite{Leo60}, see also Le Borgne \cite{Leb99} for
refinements. Other limit theorems can be found in Dolgopyat
\cite{Dol04}. The one-dimensional empirical process, for $\R$-valued
regular observables, has been studied by Durieu and Jouan
\cite{DurJou08}. Dehling and Durieu \cite{DehDur11} proved weak
convergence of the classical empirical process (indexed by indicators
of left infinite rectangles). We can now generalize this result to
empirical processes indexed by further classes of functions. We can get
the following proposition, as a corollary of Theorem \ref{thep-clt} and
the results of the preceding section.

%th4.1 #&#
%
\begin{theorem}\label{thep-torus}
Let $T$ be an ergodic d-torus automorphism and let $\F$ be one of the
following classes:
\begin{itemize}
\item the class of indicators of rectangles of $\mathbb{T}^d$;
\item the class of indicators of Euclidean balls of $\mathbb{T}^d$;
\item the class of indicators of ellipsoids of bounded diameter of
$\mathbb{T}^d$;
\end{itemize}
Then the empirical process
\[
U_n(f)=\frac{1}{\sqrt{n}}\sum_{i=1}^n
\bigl(f\circ T^i - \lambda f \bigr),\qquad f\in\F,
\]
converges in distribution in $\ell^\infty(\F)$ to a tight Gaussian
process $(W(f))_{f\in\F}$.
\end{theorem}

\begin{pf}
Let $\F$ be one of the classes of functions and $\B$ be
the class of $\alpha$-H\"older functions for some $\alpha\in(0,1]$.
We set $\G$ the subclass of $\B$ given by the functions bounded by $1$.
We consider the $\mathbb{T}^d$-valued stationary process $X_i=T^i$.
Since the distribution of $X_0$ is the Lebesgue measure on~$\mathbb{T}^d$,
Propositions \ref{prop-rect}, \ref{prop-balls} and \ref
{proellipsoidsunitsquare} show that
the condition (\ref{bracket}) holds for every possible choice of class
$\F$.
For all $f\in\B$, the central limit theorem (\ref{clt}) holds; see
Leonov \cite{Leo60} and Le Borgne~\cite{Leb99}.
Dehling and Durieu \cite{DehDur11}, Proposition 3,
show that the ergodic automorphisms of the torus satisfy the multiple mixing
property (\ref{mm}) for functions of the class $\G$, and with the
constants $\ell=1$ and $d_0$ the size of the biggest Jordan block of
$T$ restricted to its neutral subspace.
Thus, the $2p$th moment bound (\ref{2pbound}) holds, and Theorem \ref
{thep-clt} can be applied to conclude.
\end{pf}

%%%%%%%%%%%%%%%%%%%%%%%%%%%%%%%%%%%%%%%%%%%%%%%%%%%%%%%%%%%%%%%%%%%%%%%%%%%%%%%%%%%%%%%%%%%%%%%%%%%%%%%%%%%%%%%%%%%%%%%%%%%%%%%%%

%s5 #&#
\section{Proof of the main theorem}\label{secpf-ep-clt}

In the proof of Theorem \ref{thep-clt}, we need a generalization of
Theorem 4.2 of Billingsley \cite{Bil68}. Billingsley considers random variables
$X_n$, $X_n^{(m)}$, $X^{(m)}$, $X$, $m,n\geq1$, with values in a
separable metric space
$(S,\rho)$ satisfying (a) $X_n^{(m)}\cla X^{(m)}$ as $n\rightarrow
\infty$,
for all $m\geq1$, (b) $X^{(m)}\cla X$ as $m\rightarrow\infty$ and
(c) $\forall\delta>0$, $\limsup_{n\rightarrow\infty}P(\rho
(X_n^{(m)},X_n)\geq\delta)
\rightarrow0$ as $m\rightarrow\infty$. Theorem 4.2 of Billingsley
\cite{Bil68}
states that then $X_n\cla X$. Dehling, Durieu and Voln\'y \cite{DehDurVol09}
proved that this result holds without condition (b), provided that $S$ is
a complete separable metric space. More precisely, they could show that
in this situation (a)
and (c) together imply the existence of a random variable $X$
satisfying (b),
and thus by Billingsley's theorem $X_n\cla X$.
Here, we will generalize this theorem to possibly non-measurable random
elements with values in non-separable
spaces. Regarding convergence in distribution of non-measurable random
elements, we use the notation of van der Vaart and Wellner \cite{VanWel96}.
In accordance with the terminology of van der Vaart and Wellner \cite
{VanWel96}, we will call a
not necessarily measurable function with values in a measurable space a
random element.

%th5.1 #&#
%
\begin{theorem}\label{satbilerw*}
Let $X_n,X_n^{(m)},X^{(m)}$, $m,n\geq1$, be random elements with values
in a complete metric space $(S,\rho)$, and suppose that
$X^{(m)}$ is measurable and separable, that is, there is a separable
set $S^{(m)}\subset S$ such that
$\P(X^{(m)} \in S^{(m)})=1$.
If the conditions
%
%e5.1 #&#
%e5.2 #&#
%
\begin{eqnarray}
\label{vorxmconv} X_n^{(m)} &\cla& X^{(m)} \qquad\mbox{as
}n\rightarrow
\infty\mbox{, for all } m\geq1,
\\
\label{vorxnnearxnm}
\limsup_{n \rightarrow\infty} \P^* \bigl(\rho\bigl(X_n,X_n^{(m)}
\bigr)\geq\delta\bigr) &\longrightarrow& 0 \hspace*{15pt}\qquad\mbox{as } m
\rightarrow
\infty\mbox{, for all } \delta>0
\end{eqnarray}
are satisfied, then there exists an $S$-valued, separable random
variable $X$ such that
$ X^{(m)} \dconv X$ as $m\rightarrow\infty$, and
%
%e5.3 #&#
%
\[
X_n\dconv X \qquad\mbox{as }n \rightarrow\infty.
\]
\end{theorem}

The proof is postponed to the \hyperref[app]{Appendix}.

\begin{pf*}{Proof of Theorem \ref{thep-clt}}
For all $q\ge1$, there exist two sets of $N_q:=N(2^{-q},\exp(C
2^{q/\gamma}),\F,\allowbreak\G,L^s(\mu))$ functions
$\{g_{q,1},\ldots,g_{q,N_q}\}\subset\G$ and
$\{g_{q,1}',\ldots,g_{q,N_q}'\}\subset\G$, such that
$\|g_{q,i}-g_{q,i}'\|_s\le2^{-q}$, $\|g_{q,i}\|_\B\le\exp(C
2^{q/\gamma})$, $\|g_{q,i}'\|_\B\le\exp(C 2^{q/\gamma})$
and for all $f\in\F$, there exists an $i$ such that $g_{q,i}\le f\le
g_{q,i}'$. Further, by (\ref{bracket}),
%
%e5.4 #&#
%
\begin{equation}\label{eq-bracketsum}
\sum_{q\ge1}2^{-(r+1)q} N_q^2
< +\infty.
\end{equation}
For all $q\ge1$, we can build a partition $\F=\bigcup_{i=1}^{N_q}\F
_{q,i}$ of the class $\F$ into $N_q$ subsets such that
for all $f\in\F_{q,i}$, $g_{q,i}\le f\le g_{q,i}'$.
To see this, define $\F_{q,1}=[g_{q,1},g_{q,1}']$ and $\F
_{q,i}=[g_{q,i},g_{q,i}'] \setminus( \bigcup_{j=1}^{i-1} \F_j$).

In the sequel, we will use the notation $\pi_qf=g_{q,i}$ and $\pi
_q'f=g_{q,i}'$ if $f\in\F_{q,i}$.
For each $q\ge1$, we introduce the process
\[
F_n^{(q)}(f):=F_n(\pi_qf)=
\frac{1}{n}\sum_{i=1}^n
\pi_qf(X_i);\qquad f\in\F,
\]
which is constant on each $\F_{q,i}$. Further, if $f\in\F_{q,i}$, we have
\[
F_n^{(q)}(f)\le F_n(f)\le F_n\bigl(
\pi_q'f\bigr).
\]
We introduce
\[
U_n^{(q)}(f):=U_n(\pi_qf)=\sqrt{n}
\bigl( F_n^{(q)}(f)-\mu(\pi_qf)\bigr);\qquad f\in\F.
\]

%pr5.2 #&#
%
\begin{proposition}\label{prop1}
For all $q\ge1$, the sequence $(U_n^{(q)}(f))_{f\in\F}$ converges in
distribution in $\ell^\infty(\F)$ to a piecewise constant Gaussian
process $(U^{(q)}(f))_{f\in\F}$ as $n\rightarrow\infty$.
\end{proposition}
\begin{pf}
Since $\pi_q f\in\G$ and $\G$ is a subset of $\B$, by assumption
(\ref{clt}), the CLT holds and $U_n^{(q)}(f)$ converges to a Gaussian
law for all $f\in\F$. We can apply the Cram\'er--Wold device to get
the finite-dimensional convergence: for all $k\ge1$, for all
$f_1,\ldots,f_k\in\F$, $(U^{(q)}_n(f_1),\ldots,U^{(q)}_n(f_k))$
converges in distribution to a Gaussian vector
$(U^{(q)}(f_1),\ldots,U^{(q)}(f_k))$ in $\R^k$. Since $U_n^{(q)}$ is
constant on each element
$\F_{q,i}$ of the partition, the finite-dimensional convergence
implies the weak convergence of the process.
Indeed, consider the function $\tau_q\dvtx\R^{N_q}\rightarrow\ell
^\infty(\F)$ that maps a vector $x=(x_1,\ldots,x_{N_q})$ to the
function $\F\rightarrow\R$, $f\mapsto x_i$ such that $f\in\F_{q,i}$.
For $f_1\in\mathcal{F}_{q,1},\ldots,f_{N_q}\in\mathcal{F}_{q,N_q}$
we have $U_n^{(q)}=\tau_q (U_n^{(q)}(f_1),\ldots,U_n^{(q)}(f_{N_q}))$
and thus the continuous mapping theorem guarantees that $U_n^{(q)}$
converges weakly to the random variable $U^{(q)}=\tau_q
(U^{(q)}(f_1),\ldots,U^{(q)}(f_{N_q}))$ which is constant on each $\F_{q,i}$.
\end{pf}

%pr5.3 #&#
%
\begin{proposition}\label{prop2}
For all $\varepsilon>0$, $\eta>0$ there exists a $q_0$ such that for
all $q\geq q_0$
\[
\limsup_{n\rightarrow\infty}\P^*\Bigl(\sup_{f\in\F
}\bigl|U_n(f)-U_n^{(q)}(f)\bigr|>
\varepsilon\Bigr)\le\eta.
\]
\end{proposition}

\begin{pf}
For a random variable $Y$ let $\overline{Y}$ denote its centering
$\overline{Y}:= Y - \E Y$.
If for arbitrary random variables $Y_l, Y, Y_u$ we have $Y_l \le Y \le
Y_u$, then
\[
| \overline{Y} - \overline{Y}_l | \leq|\overline{Y}_u
- \overline{Y}_l| + \E|Y_u - Y_l|.
\]
Using $F_n^{(q+K)}(f) \leq F_n(f) \leq F_n(\pi_{q+K}'f)$ and $\E
|F_n(\pi_{q+K}'f) - F_n^{(q+K)}(f)| \leq2^{-(q+K)}$
for all $f\in\F$, we obtain
\begin{eqnarray*}
\bigl|U_n(f) - U_n^{(q)}(f)\bigr| &=& \Biggl|   \sum
^K_{k=1}\bigl(U_n^{(q+k)}(f)-U_n^{(q+k-1)}(f)
\bigr) + U_n(f) - U_n^{(q+K)}(f) \Biggr|
\\
&\leq&   \sum^K_{k=1} \bigl|
U_n^{(q+k)}(f)-U_n^{(q+k-1)}(f) \bigr| +
\bigl|U_n\bigl(\pi_{q+K}'f\bigr)-U_n^{(q+K)}(f)
\bigr|
\\
&&{}+ \sqrt{n}2^{-(q+K)}.
\end{eqnarray*}
In order to assure $\frac{\varepsilon}{4}\le2^{-(q+{K})}\sqrt{n}\le
\frac{\varepsilon}{2}$,
for fixed $n$ and $q$, choose $K=K_{n,q}$, where
\[
K_{n,q}:= \biggl\lfloor\log\biggl(\frac{4\sqrt{n}}{2^q\varepsilon
} \biggr)
\log(2)^{-1} \biggr\rfloor.
\]
For each $i\in\{1,\ldots,N_q\}$, we obtain
\begin{eqnarray*}
\sup_{f\in\F_{q,i}}\bigl| U_n(f) - U_n^{(q)}(f)\bigr|
&\le& \sum^{K_{n,q}}_{k=1}\sup
_{f\in\F
_{q,i}}\bigl|U_n^{(q+k)}(f)-U_n^{(q+k-1)}(f)\bigr|
\\
&&{}+\sup_{f\in\F_{q,i}}\bigl|U_n\bigl(\pi_{q+{K_{n,q}}}'f
\bigr)-U_n^{(q+{K_{n,q}})}(f)\bigr| +\frac{\varepsilon}{2}.
\end{eqnarray*}
By taking $\varepsilon_k=\frac{\varepsilon}{4k(k+1)}$, $\sum_{k\ge
1}\varepsilon_k=\frac{\varepsilon}{4}$ and we get for each $i\in\{
1,\ldots,N_q\}$,
\begin{eqnarray*}
\P^* \Bigl(\sup_{f\in\F_{q,i}}\bigl|U_n(f)-U_n^{(q)}(f)\bigr|
\ge\varepsilon\Bigr)&\le& \sum^{K_{n,q}}_{k=1}
\P^* \Bigl(\sup_{f\in\F_{q,i}}\bigl|U_n^{(q+k)}(f) -
U_n^{(q+k-1)}(f)\bigr|\ge\varepsilon_k \Bigr)
\\
&&{}+\P^* \biggl(\sup_{f\in\F_{q,i}}\bigl|U_n\bigl(
\pi_{q+{K_{n,q}}}'f\bigr) - U_n^{(q+{K_{n,q}})}(f)\bigr|\ge
\frac{\varepsilon}{4} \biggr).
\end{eqnarray*}
Notice that the suprema in the r.h.s. are in fact maxima over finite
numbers of functions, since the functionals $\pi_q$ and $\pi_q'$ (and
thus $U_n^{(q)}$) are constant on the $\mathcal{F}_{q,i}$.
Therefore, we can work with standard probability theory from this
point: the outer probabilities can be replaced by usual probabilities
on the right-hand side.
For each $k$, choose a set $F_{k}$ composed by at most $N_{k-1}N_k$
functions of $\F$ in such a way that $F_k$ contains one function in
each non-empty
$\F_{k-1,i}\cap\F_{k,j}$, $i=1,\ldots,N_{k-1}$, $j=1,\ldots,N_{k}$.
Then, for each $i\in\{1,\ldots,N_q\}$, we have
\begin{eqnarray*}
&&
\P^* \Bigl(\sup_{f\in\F_{q,i}}\bigl|U_n(f)-U_n^{(q)}(f)\bigr|
\ge\varepsilon\Bigr)\\[-2pt]
&&\quad\le\sum^{K_{n,q}}_{k=1}
\sum_{f\in\F_{q,i}\cap F_{q+k}} \P\bigl(\bigl|U_n^{(q+k)}(f)
- U_n^{(q+k-1)}(f)\bigr|\ge\varepsilon_k \bigr)
\\[-2pt]
&&\qquad{}+\sum_{f\in\F_{q,i}\cap F_{q+{K_{n,q}}}} \P\biggl(\bigl
|U_n\bigl(
\pi_{q+{K_{n,q}}}'f\bigr) - U_n^{(q+{K_{n,q}})}(f)\bigr|\ge
\frac
{\varepsilon}{4} \biggr).
\end{eqnarray*}
Now using Markov's inequality at the order $2p$ ($p$ will be chosen
later) and assumption (\ref{2pbound}), we infer
\begin{eqnarray*}
&&
\P^* \Bigl(\sup_{f\in\F_{q,i}}\bigl|U_n(f)-U_n^{(q)}(f)\bigr|
\ge\varepsilon\Bigr)
\\[-2pt]
&&\quad\le C_p\sum^{K_{n,q}}_{k=1}\sum
_{f\in\F_{q,i}\cap F_{q+k}}\frac
{1}{\varepsilon_k^{2p}}\sum
_{j=1}^pn^{j-p}\|\pi_{q+k}f-\pi
_{q+k-1}f\|_s^j\\[-2pt]
&&\qquad\hspace*{108pt}{}\times\log^{2p+aj}\bigl(\|
\pi_{q+k}f-\pi_{q+k-1}f\|_{\B}+1\bigr)
\\[-2pt]
&&\qquad{}+ C_p \sum_{f\in\F_{q,i}\cap F_{q+{K_{n,q}}}} \biggl(
\frac{4}{\varepsilon} \biggr)^{2p}\sum_{j=1}^pn^{j-p}
\bigl\|\pi_{q+{K_{n,q}}}f-\pi_{q+{K_{n,q}}}'f\bigr\|_s^j\\[-2pt]
&&\qquad\hspace*{127pt}{}\times
\log^{2p+aj}\bigl(\bigl\|\pi_{q+{K_{n,q}}}f-\pi_{q+{K_{n,q}}}'f
\bigr\|_{\B}+1\bigr).
\end{eqnarray*}
At this point, without loss of generality, we can assume that $a\ge-1$
(if not, take a larger $a$) and thus the assumption on $\gamma$
reduces to $\gamma>2+a$.

Note that by construction, for each $k\ge1$,
\begin{eqnarray*}
\|\pi_{q+k}f-\pi_{q+k-1}f\|_s &\le& \|
\pi_{q+k}f-f\|_s+\|\pi_{q+k-1}f-f\|_s
\le3 \cdot2^{-(q+k)},
\\[-2pt]
\bigl\|\pi_{q+k}f-\pi_{q+k}'f\bigr\|_s &\le&
2^{-(q+k)},
\\[-2pt]
\|\pi_{q+k}f-\pi_{q+k-1}f\|_{\B} &\le& 2\exp
\bigl(C2^{({q+k})/\gamma}\bigr),
\\[-2pt]
\bigl\| \pi_{q+k}f-\pi_{q+k}'f\bigr\|_{\B} &
\le& 2\exp\bigl(C2^{({q+k})/\gamma}\bigr).
\end{eqnarray*}
Thus,
\begin{eqnarray*}
&&\P^* \Bigl(\sup_{f\in\F
_{q,i}}\bigl|U_n(f)-U_n^{(q)}(f)\bigr|
\ge\varepsilon\Bigr)
\\
&&\quad\le2^{2p+1}C_p\sum_{j=1}^p
\sum^{K_{n,q}}_{k=1}\#(\F_{q,i}\cap
F_{q+k})\frac{(k(k+1))^{2p}}{\varepsilon^{2p}}\\
&&\hspace*{66.2pt}\qquad{}\times n^{j-p}2^{-j(q+k)}\log
^{2p+aj}\bigl(2\exp\bigl(C2^{({q+k})/\gamma}\bigr)+1\bigr)
\end{eqnarray*}
and if $q$ is large enough,
\begin{eqnarray*}
&&\P^* \Bigl(\sup_{f\in\F}\bigl|U_n(f)-U_n^{(q)}(f)\bigr|
\ge\varepsilon\Bigr)
\\
&&\quad\le\sum_{i=1}^{N_q}\P^* \Bigl(\sup
_{f\in\F
_{q,i}}\bigl|U_n(f)-U_n^{(q)}(f)\bigr|\ge
\varepsilon\Bigr)
\\
&&\quad\le D\sum_{i=1}^{N_q}\sum
_{j=1}^p\sum^{K_{n,q}}_{k=1}
\#(\F_{q,i}\cap F_{q+k})\frac{(k(k+1))^{2p}}{\varepsilon
^{2p}}n^{j-p}2^{-j(q+k)}2^{(2p+aj)({{q+k}})/{\gamma}},
\end{eqnarray*}
where $D$ is a new constant which depends on $p$, $C$, and $C_p$.
Since $(\F_{q,i})_{i=1,\ldots,N_q}$ is a partition of $\F$, we have
\[
\sum_{i=1}^{N_q}\#(\F_{q,i}\cap
F_{q+k})=\#(F_{q+k})\le N_{q+k-1}N_{q+k},
\]
thus we have
%
%e5.5 #&#
%
\begin{eqnarray}\label{eqlast}
&&
\P^* \Bigl(\sup_{f\in\F}\bigl|U_n(f)-U_n^{(q)}(f)\bigr|
\ge\varepsilon\Bigr)
\nonumber
\\
&&\quad\le D'\sum_{j=1}^p
\frac{n^{j-p}}{\varepsilon^{2p}} \sum
^{K_{n,q}}_{k=1}N_{q+k-1}N_{q+k}k^{4p}2^{(2p+(a-\gamma)j)
({q+k})/{\gamma}}
\\
&&\quad\le D'\sum_{j=1}^{p}
\frac{n^{j-p}}{\varepsilon^{2p}} 2^{(p-j)(\gamma+2+a)
({q+{K_{n,q}}})/{\gamma}}\nonumber\\
&&\hspace*{26pt}\qquad{}\times\sum^{K_{n,q}}_{k=1}N_{q+k-1}N_{q+k}
k^{4p}2^{((-a-\gamma)p+(2+2a)j)({q+k})/{\gamma}}
\nonumber
\\
&&\quad\le D''\sum_{j=1}^{p-1}
\frac{n^{(j-p)({\gamma-(2+a)})/({2\gamma
})}}{\varepsilon^{2p+(p-j)({\gamma+2+a})/{\gamma}}}\sum^\infty
_{k=1}N_{q+k-1}N_{q+k}
k^{4p}2^{(2+a-\gamma)p({q+k})/{\gamma}}\nonumber\\
&&\qquad{} + \frac{D'}{\varepsilon^{2p}}
\sum
^\infty_{k=1}N_{q+k-1}N_{q+k}
k^{4p}2^{(2+a-\gamma)p({q+k})/{\gamma}},\nonumber
\end{eqnarray}
because $a\ge-1$ and thus $(2+2a)j\le(2+2a)p$, and where $D'$ and
$D''$ are positive constants also depending on $p$, $C$, and $C_p$.
As $p\frac{2+a-\gamma}{\gamma}\rightarrow-\infty$ when $p$ tends
to infinity, there exists some $p>1$ such that $p\frac{2+a-\gamma
}{\gamma} < -(r+1)$\label{condonp} and thus by (\ref{eq-bracketsum}),
\[
\sum^\infty_{k=2}N_{k-1}N_k
k^{4p}2^{p(2+a-\gamma){k}/{\gamma}} \le\sum^\infty_{k=2}N_{k-1}^2
k^{4p}2^{p(2+a-\gamma){k}/{\gamma}}+ \sum^\infty_{k=2}N_k^2
k^{4p}2^{p(2+a-\gamma){k}/{\gamma}} <+\infty.
\]
Therefore, the first summand of (\ref{eqlast}) goes to zero as $n$
goes to infinity and the second summand of (\ref{eqlast}) goes to
zero as $q$ goes to infinity.
\end{pf}

Propositions \ref{prop1} and \ref{prop2} establish for the random
elements $U_n$, $U_n^{(q)}$, $U^{(q)}$ with
value in the complete metric space $\ell^\infty(\F)$ conditions
(\ref{vorxmconv}) and (\ref{vorxnnearxnm})
of Theorem \ref{satbilerw*}, respectively. Thus, Theorem \ref
{satbilerw*} completes the proof of Theorem \ref{thep-clt}.
\end{pf*}

%%%%%%%%%%%%%%%%%%%%%%%%%%%%%%%%%%%%%%%%%%%%%%%%%%%%%%%%%%%%%%%%%%%%%%%%%%%%%%%%%%%
%%%%%%%%%%%%%%%%%%%%%%%%%%%%%%%%%%%%%%%%%%%%%%%%%%%%%%%%%%%%%%%%%%%%%%%%%%%%%%%%%%%

\begin{appendix}\label{app}
\section*{Appendix}

\begin{pf*}{Proof of Lemma \ref{lemtransition}}
By the triangle inequality, we have for all $x,y\in\X$ that
\begin{eqnarray*}
\bigl|d_B(x) - d_B(y)\bigr| &\leq& d(x,y),
\\
d_B(x) + d_A(y) &\leq& d(A,B).
\end{eqnarray*}
Therefore,
\begin{eqnarray*}
&&
\bigl|T[A,B](x) - T[A,B](y)\bigr|
\\
&&\quad= \biggl\llvert\frac{ (d_B(x)-d_B(y) )
(d_B(y)+d_A(y) ) + d_B(y) (d_B(y)+d_A(y) )-d_B(y)
(d_B(x)+d_A(x) )}{ (d_B(x)+d_A(x) )
(d_B(y)+d_A(y) )} \biggr\rrvert
\\
&&\quad= \biggl\llvert\frac{d_B(x)-d_B(y)}{d_B(x)+d_A(x)} \biggr
\rrvert
+ \frac{d_B(y)}{d_B(y)+d_A(y)}
\biggl\llvert\frac{ ( d_B(y)-d_B(x) ) + (d_A(y) -
d_A(x) )}{d_B(x)+d_A(x)} \biggr\rrvert
\\
&&\quad\leq3 \frac{d(x,y)}{d(A,B)}
\end{eqnarray*}
and thus
\begin{eqnarray*}
\bigl\|T[A,B]\bigr\|_\alpha&:=& \bigl\|T[A,B]\bigr\|_\infty+ \sup
_{x\ne y} \frac{
|T[A,B](x)-T[A,B](y) |}{d(x,y)^\alpha}
\\
&\leq& 1 + \sup_{x\ne y} \biggl( \frac{
|T[A,B](x)-T[A,B](y) |}{d(x,y)}
\biggr)^\alpha\bigl| T[A,B](x)-T[A,B](y) \bigr|^{1-\alpha}
\\
&\leq& 1 + \biggl( \frac{3}{d(A,B)} \biggr)^\alpha.
\end{eqnarray*}
\upqed\end{pf*}

\begin{pf*}{Proof of Lemma \ref{lemdistellip}}
Without loss of generality, assume that $x=0$.
For $v\in\R^d$, let $\mathrm{D}_{v}$ denote the diagonal $d \times
d$-matrix with diagonal entries $v_1,\ldots,v_d$.
We define the operator norm of the $d \times d$-matrix $A$ by
$|A|_*:= \sup_{y\in\R^d\setminus\{0\}} |Ay|/|y|$.
Observe that\vspace*{1pt} $|\mathrm{D}_{v}|_*=\max_{i=1,\ldots,d} |v_i|$.
We can characterize $\ellip(0,\frac{j}{m})$ and $\R^d\setminus
\ellip(0,\frac{j}{m}+\frac{1}{m})$ by
\[
\ellip\biggl(0,\frac{j}{m} \biggr) = \bigl\{z\in\R^d\dvt
\bigl\llvert\mathrm{D}_{{j}/{m}}^{-1} z\bigr\rrvert\leq1 \bigr
\}
\]
and
\[
\R^d\setminus\ellip\biggl(0,\frac{j}{m}+
\frac{1}{m} \biggr) = \bigl\{y\in\R^d\dvt \bigl\llvert\mathrm
{D}_{{j}/{m}+{1}/{m}}^{-1} y\bigr\rrvert> 1 \bigr\},
\]
respectively. Thus, for any $z\in\ellip(0,\frac{j}{m} )$
and $y\in\R^d\setminus\ellip(0,\frac{j}{m}+\frac
{1}{m} )$,
\begin{eqnarray*}
|y-z| &\ge& \bigl\llvert\mathrm{D}_{{j}/{m}+{1}/{m}}^{-1}
\bigr\rrvert_*^{-1} \bigl\llvert\mathrm{D}_{{j}/{m}+
{1}/{m}}^{-1} y - \mathrm{D}_{{j}/{m}+ {1}/{m}}^{-1}\mathrm
{D}_{{j}/{m}}\mathrm{D}_{{j}/{m}}^{-1} z \bigr\rrvert
\\
&\ge& \bigl\llvert\mathrm{D}_{{j}/{m}+{1}/{m}}^{-1}
\bigr\rrvert_*^{-1} \bigl(\bigl\llvert\mathrm{D}_{
{j}/{m}+ {1}/{m}}^{-1} y\bigr\rrvert- \bigl\llvert\mathrm
{D}_{{j}/{m}+{1}/{m}}^{-1}\mathrm{D}_{{j}/{m}}\bigr
\rrvert_* \bigl\llvert\mathrm{D}_{{j}/{m}}^{-1} z \bigr
\rrvert
\bigr)
\\
&>& \bigl\llvert\mathrm{D}_{{j}/{m}+{1}/{m}}^{-1}
\bigr\rrvert_*^{-1} \bigl( 1 - \bigl\llvert\mathrm{D}_{
{j}/{m}+ {1}/{m}}^{-1} \mathrm{D}_{{j}/{m}} \bigr\rrvert_*
\bigr)
\\
&=& \min_{j_i=1,\ldots,d} \biggl\{\frac{j}{m}+\frac{1}{m}
\biggr\} \biggl( 1 - \max_{i=1,\ldots,d} \biggl\{\frac{
{j_i}/{m} }{
{j_i}/{m} + {1}/{m} }
\biggr\} \biggr)
\\
&\geq& \frac{1}{Dm^2}
\end{eqnarray*}
since $j_i\in\{0,\ldots,Dm-1\}$.
\end{pf*}

\begin{pf*}{Proof of Lemma \ref{lemextension-lemma}}
For any $\varepsilon>0$, set $K_{\varepsilon}=\sup\{ K > 0\dvt \mu
([-K,K]^d) \leq1 - \varepsilon\}$. We will denote the function
$(0,1)\rightarrow\R^+$, $\varepsilon\mapsto K_{\varepsilon}$ by
$K_{\bullet}$.
Now, introduce the bracket $[L, U_{\varepsilon}]$, given by
\[
L\equiv0 \quad\mbox{and}\quad U_{\varepsilon}:= T \bigl[\R^d
\setminus[-K_{\varepsilon
^s/2},K_{\varepsilon^s/2}]^d,
[-K_{\varepsilon^s},K_{\varepsilon
^s}]^d \bigr].
\]
Obviously, we have $\|U_{\varepsilon} - L\|_s\le\|U_{\varepsilon} -
L\|_1^{1/s} \leq\varepsilon$.

To get a bound for the H\"older-norm of $U_{\varepsilon}$,
consider the distribution function
\[
G(t):= \mu\bigl( \bigl\{ x\in\R^d\dvt | x |_{\max} \leq t
\bigr\} \bigr)
\]
on $\R$, where $| x |_{\max}=\max\{|x_i|\dvt i=1,\ldots,d\}$.
Observe that the pseudo-inverse $G^{-1}$ of $G$
is linked to $K_{\bullet}$ by the equality $ K_{\varepsilon} =
G^{-1}(1-\varepsilon)$.
With geometrical arguments, we infer
\[
G(t) = \sum_{j\in\{-1,1\}^d} \sigma(j) F(tj),
\]
where $\sigma(j):= \prod_{i=1}^d j_i \in\{-1,1\}$. Therefore,
\begin{eqnarray*}
\omega_G(x) &=& \sup_{t\in\R} \bigl\{ G(t+x)- G(t)
\bigr\} = \sup_{t\in\R} \sum_{j\in\{-1,1\}^d}
\sigma(j) \bigl( F\bigl((t+x)j\bigr) - F(tj) \bigr)
\\
&\leq& \sum_{j\in\{-1,1\}^d} \sup_{t\in\R} \bigl| F
\bigl((t+x)j\bigr) - F(tj) \bigr| \leq\sum_{j\in\{-1,1\}^d}
\omega_F(\sqrt{d} x)
\\
&\leq& 2^d \omega_F(\sqrt{d} x).
\end{eqnarray*}
Now by Lemma \ref{lemtransition} we obtain
\begin{eqnarray*}
\| U_{\varepsilon} \|_{\alpha} &\leq& 1 + \frac{3^{\alpha
}}{|G^{-1}(1-{\varepsilon^s}/{2}) -
G^{-1}(1-{\varepsilon^s})|^{\alpha}}
\\
&\leq& 1 + 3^{\alpha} \biggl(\inf\biggl\{ x>0\dvt \exists t\in\R
\mbox{ such
that } G(t+x)-G(t) \geq\frac{\varepsilon^s}{2} \biggr\} \biggr
)^{-\alpha
}
\\
&\leq& 1 + 3^{\alpha} \biggl(\inf\biggl\{ x>0\dvt \omega_G(x)
\geq\frac
{\varepsilon^s}{2} \biggr\} \biggr)^{-\alpha}
\\
&\leq& 1 + 3^{\alpha} \biggl(\sup\biggl\{ x\ge0\dvt \omega_F(
\sqrt{d}x) \leq\frac{\varepsilon^s}{2^{d+1}} \biggr\} \biggr
)^{-\alpha
}
\\
&=& 1 + (3\sqrt{d})^{\alpha} \bigl(\omega_F^{-1}
\bigl(2^{-(d+1)}\varepsilon^s\bigr)\bigr)^{-\alpha},
\end{eqnarray*}
where we used that $\omega_F$ is continuous here to replace the
infimum by the supremum.

Then $[L, U_{\varepsilon}]$ is an $(\varepsilon, 4\sqrt{d}
(\omega_F^{-1}(2^{-(d+1)}\varepsilon^s))^{-\alpha},\G,L^s(\mu
))$-bracket for sufficiently small $\varepsilon$.
Since $[L, U_{\varepsilon}]$ contains any $f\in\mathcal
{F}\setminus\F_{K_{\varepsilon/2}+D}$, by (\ref{eqtemp100}) we
obtain for all those $\varepsilon$ the bound
\[
N \bigl(\varepsilon,\max\bigl\{f(\varepsilon), 4\sqrt{d}\bigl
(\omega
_F^{-1}\bigl(2^{-(d+1)}\varepsilon^s
\bigr)\bigr)^{-\alpha} \bigr\},\F,\G,L^s(\mu) \bigr) \leq C
(K_{\varepsilon^s/2}+D)^p \varepsilon^{-q} + 1.
\]
Let us finally consider the growth rate of $K_{\varepsilon^s/2}$ as
$\varepsilon\rightarrow0$.
By assumption (\ref{eqtmpa03}) and since \mbox{$|\cdot|_{\max} \le
|\cdot|$}, we have $1-G(t)\le bt^{-1/\beta}$ for sufficiently large
$t$. Therefore,
\[
G\bigl((b /\varepsilon)^{\beta}\bigr) \ge1- {\varepsilon}.
\]
By the definition of $K_{\bullet}$, we therefore obtain that
$K_{\varepsilon^s/2} \leq(2b/\varepsilon^{s})^{\beta} =
\RMO_{\beta,b}(\varepsilon^{-\beta s})$ which proves the lemma.
\end{pf*}

\begin{pf*}{Proof of Theorem \ref{satbilerw*}}
(i) We will first show that $X^{(m)}$ converges in distribution to some
random variable $X$. We denote by $L^{(m)}$ the distribution of $X^{(m)}$;
this is defined
since $X^{(m)}$ is measurable. Moreover, $L^{(m)}$ is a separable
Borel probability measure on $S$. By Theorem 1.12.4 of van der Vaart
and Wellner \cite{VanWel96}, weak convergence of separable Borel
measures on a metric
space $S$ can be metrized by the bounded Lipschitz metric, defined by
%
%e5.6 #&#
%
\setcounter{equation}{0}
\[
d_{\mathrm{B L}_1} (L_1,L_2)=\sup_{f\in{\mathrm{B L}_1}}
\biggl\llvert\int f(x) \mrmd L_1(x) -\int f(x) \mrmd L_2(x)\biggr
\rrvert
\]
for any Borel measures $L_1, L_2$ on $S$. Here,
$ \mathrm{B L}_1:=\{f\dvtx S\lra\R\dvt \|f\|_{\mathrm{B L}_1} \leq1\}$,
where
\[
\|f\|_{\mathrm{B L}_1}:=\max\biggl\{\sup_{x\in S}\bigl|f(x)\bigr|, \sup
_{x\neq y\in S} \frac{f(x)-f(y)}{\rho(x,y)} \biggr\}.
\]
In addition, the theorem states that the space of all separable Borel
measures on a complete space is
complete with respect to the bounded Lipschitz metric. Thus, it suffices
to show that $L^{(m)}$ is a $d_{\mathrm{B L}_1}$-Cauchy sequence. We obtain
\begin{eqnarray*}
d_{\mathrm{B L}_1}\bigl(L^{(m)},L^{(l)}\bigr) &=& \sup
_{f\in{\mathrm{B L}_1}}\bigl|\E f\bigl(X^{(m)}\bigr)-\E f
\bigl(X^{(l)}\bigr)\bigr|
\\
&\leq&\sup_{f\in{\mathrm{B L}_1}} \bigl\{ \bigl|\E f\bigl(X^{(m)}\bigr)-
\E^* f\bigl(X_n^{(m)}\bigr)\bigr| + \bigl|\E^* f\bigl(X_n^{(m)}
\bigr)-\E^* f(X_n)\bigr|
\\
&&\hspace*{27.5pt}{} + \bigl|\E^* f(X_n)-\E^* f\bigl(X_n^{(l)}\bigr)\bigr| +
\bigl|\E^* f\bigl(X_n^{(l)}\bigr)-\E f\bigl(X^{(l)}
\bigr)\bigr| \bigr\}
\end{eqnarray*}
for all $n\in\N$. For a Borel measurable separable random element
$X^{(m)}$ weak convergence $X_n^{(m)}\dconv X^{(m)}$ as $n \rightarrow
\infty$ is equivalent to\vspace*{1pt}
$\sup_{f\in{\mathrm{B L}_1}}|\E f(X^{(m)})-\E^* f(X_n^{(m)})|\lra0$; see
van der Vaart and Wellner \cite{VanWel96}, page 73.
Hence by (\ref{vorxmconv}), we obtain
\[
d_{\mathrm{B L}_1}\bigl(L^{(m)},L^{(l)}\bigr) \leq\liminf
_{n \rightarrow\infty} \sup_{f\in{\mathrm{B L}_1}} \bigl|\E^* f
\bigl(X_n^{(m)}\bigr)-\E^* f(X_n)\bigr| + \bigl|\E^*
f(X_n)-\E^* f\bigl(X_n^{(l)}\bigr)\bigr|.
\]
Using Lemma 1.2.2(iii) in van der Vaart and Wellner \cite{VanWel96},
we obtain
\[
\bigl|\E^* f\bigl(X_n^{(m)}\bigr)-\E^* f(X_n)\bigr| \leq
\E\bigl(\bigl| f(X_n)- f\bigl(X_n^{(m)}
\bigr)\bigr|^\ast\bigr)\vadjust{\goodbreak}
\]
and therefore
%
%e5.7 #&#
%
\begin{eqnarray}\label{eqtemp001}
\sup_{f\in{\mathrm{B L}_1}} \bigl|\E^* f\bigl(X_n^{(m)}\bigr)-
\E^* f(X_n)\bigr| &\leq& \E\bigl(\rho\bigl(X_n,X_n^{(m)}
\bigr) \wedge2 \bigr)^*
\nonumber\\[-8pt]\\[-8pt]
&=&\int_0^{\infty} \P^\ast\bigl(\rho
\bigl(X_n,X_n^{(m)}\bigr) \wedge2 \geq t \bigr)
\mrmd t,\nonumber
\end{eqnarray}
where we used the last statement of Lemma 1.2.2 in van der Vaart and
Wellner \cite{VanWel96}.
Now, let $\epsilon>0$ be given. By (\ref{vorxnnearxnm}), there
exists an $m_0\in\N$ such that for every $m\geq m_0$ there is some
$n_0\in\N$ such that for every $n\geq n_0$ we have
$ \P^* (\rho(X_n, X_n^{(m)}) \geq{\epsilon}/{3} ) \leq
{\epsilon}/{3}$. Therefore,
\[
\P^* \bigl(\rho\bigl(X_n, X_n^{(m)}\bigr)
\wedge2 \geq t \bigr) \leq%
\cases{1, &\quad if $\displaystyle t<\frac{\epsilon}{3}$,
\vspace*{2pt}\cr
\displaystyle \frac{\epsilon}{3}, &\quad if $\displaystyle \frac{\epsilon}{3}\leq t \leq2$,
\vspace*{3pt}\cr
0, &\quad if $2<t$.}
\]
Applying this inequality to (\ref{eqtemp001}), we obtain
\[
\liminf_{n \rightarrow\infty} \sup_{f\in{\mathrm{B L}_1}} \bigl|\E^* f
\bigl(X_n^{(m)}\bigr)-\E^* f(X_n)\bigr| \leq\int
_0^{2} \frac{\epsilon}{3} + 1_{\{t<{\epsilon}/{3}\}}
\mrmd t = \epsilon
\]
for all $m\geq m_0$. Hence for $l,m\geq m_0$ we have $d_{\mathrm{B
L}_1}(L^{(m)},L^{(l)})\leq2 \epsilon$; that is, $(L^{(m)})_{m\in\N
}$ is a $d_{\mathrm{B L}_1}$-Cauchy sequence in a complete metric space.

(ii) The remaining part of the proof follows closely the proof of
Theorem 4.2 in Billingsley~\cite{Bil68}, replacing the probability
measure $\P$
by the outer measure $\P^*$ where necessary and making use of the
Portmanteau theorem; see van der Vaart and Wellner \cite{VanWel96},
Theorem 1.3.4(iii), and
the sub-additivity of outer measures. From part (i), we already know
that there is some measurable $X$ such that $X^{(m)}\dconv X$. Let
$F\subset S$ be closed. Given $\epsilon>0$, we define the $\epsilon
$-neighborhood
$ F_{\epsilon}:=\{s\in S\dvt \inf_{x\in F} \rho(s,x)\leq\epsilon\}$,
and observe that $F_{\epsilon}$ is also closed.
Since $ \{ X_n\in F\}\subset\{X_n^{(m)}\in F_{\epsilon}\} \cup
\{ \rho(X_n^{(m)},X_n) \geq\epsilon\}$,
we obtain
\[
\P^*(X_n\in F) \leq\P^*\bigl(X_n^{(m)}\in
F_{\epsilon}\bigr)+ \P^*\bigl(\rho\bigl(X_n^{(m)},X_n
\bigr)\geq\epsilon\bigr)
\]
for all $m\in\N$. By (\ref{vorxnnearxnm}) we may choose $m_0$ so
large that for all $m\geq m_0$
\[
\limsup_{n \rightarrow\infty}\P^*\bigl(\rho\bigl(X_n^{(m)},X_n
\bigr)\geq\epsilon\bigr)\leq\epsilon/2.
\]
As $X^{(m)}\dconv X$, by the Portmanteau theorem we may choose $m_1$ so
large that for all $m\geq m_1$
\[
\P\bigl(X^{(m)}\in F_{\epsilon}\bigr)\leq\P(X\in
F_{\epsilon})+\epsilon/2.
\]
We now fix $m\geq\max(m_0,m_1)$. By (\ref{vorxmconv}) we have
$X_n^{(m)} \dconv X^{(m)}$ as $n\rightarrow\infty$. Thus an
application of the Portmanteau theorem yields
\begin{eqnarray*}
\limsup_{n \rightarrow\infty}\P^*\bigl(X_n^{(m)}\in
F_{\epsilon}\bigr) &\leq&\P\bigl(X^{(m)}\in F_{\epsilon}\bigr),
\\
\limsup_{n \rightarrow\infty}\P^*(X_n\in F) &\leq&\P(X\in
F_{\epsilon}) + \epsilon.
\end{eqnarray*}
Since this holds for any $\epsilon>0$ and $\lim_{\epsilon\rightarrow
0} \P(X\in F_\epsilon) = \P(X\in F)$,
we get
\[
\limsup_{n\rightarrow\infty} \P^*(X_n\in F) \leq\P(X\in F)
\]
for all closed sets $F\subset S$.
By a final application of the Portmanteau theorem we infer \mbox{$X_n \dconv X$}.
\end{pf*}
\end{appendix}

% zodis "Acknowledgments" paliekamas pagal autoriu
\section*{Acknowledgments}

The authors thank the referees and the Associate Editor for their very
careful reading and for their most constructive criticism of an earlier
version of this paper that greatly helped to improve the presentation.

Research supported in part by German Research Foundation Grant DE
370-4 \textit{New Techniques for Empirical Processes of Dependent Data},
and the Collaborative Research Grant SFB 823 \textit{Statistical Modeling
of Nonlinear Dynamic Processes.}

%suskaldyti doi

% imsref loaded by lrinkeviciute, 2013-07-05 14:23:20
%

\printhistory

\end{document}